\documentclass[reqno,a4paper,10pt]{amsart}
\usepackage[T1]{fontenc}
\usepackage{lmodern}
\usepackage[libertine]{newtxmath}
\usepackage[utf8]{inputenc}
\usepackage[english]{babel}
\usepackage{url}
\usepackage[colorlinks=true, linkcolor = blue, urlcolor=black, citecolor=blue, anchorcolor=blue]{hyperref}
\usepackage{comment}
\usepackage{xcolor,color}
\usepackage{geometry}
\usepackage[all]{xy}
\usepackage{enumitem}
\usepackage{booktabs}
\usepackage{slashed}
\usepackage{bbm}
\usepackage{cancel}
\usepackage{makebox}

\usepackage{cleveref}
\usepackage{tikz}
\usepackage{tikz-cd}

\usepackage{soul}

\usepackage{textcomp}

\usepackage{amsmath}
\usepackage{graphicx}
\usepackage{amsthm}
\usepackage{latexsym}
\usepackage{amsfonts}
\usepackage{bbm}
\usepackage{mathrsfs}

\numberwithin{equation}{section}

\theoremstyle{plain}
\newtheorem{lemma}{Lemma}[section]
\newtheorem{proposition}[lemma]{Proposition}
\newtheorem{proposition/definition}[lemma]{Proposition/Definition}
\newtheorem{theorem}[lemma]{Theorem}

\newtheorem{corollary}[lemma]{Corollary}

\theoremstyle{definition}
\newtheorem{definition}[lemma]{Definition}
\newtheorem{remark}[lemma]{Remark}
\newtheorem{example}[lemma]{Example}

\DeclareMathOperator{\id}{id}
\DeclareMathOperator{\im}{im}

\DeclareMathOperator{\Hom}{Hom}

\DeclareMathOperator{\gr}{Gr}

\DeclareMathAlphabet{\mathpzc}{OT1}{pzc}{m}{it}

\newcommand{\good}{\text{$\calF$}}

\newcommand{\SymplFol}{\text{$\mathpzc{SymplFol}$}}
\newcommand{\RegPoiss}{\text{$\mathpzc{RegPoiss}$}}

\newcommand{\bfv}{\mathbf{v}}

\newcommand{\calF}{\mathcal{F}}

\newcommand{\calI}{\mathcal{I}}

\newcommand{\calL}{\mathcal{L}}

\newcommand{\calQ}{\mathcal{Q}}
\newcommand{\calR}{\mathcal{R}}

\newcommand{\bbN}{\mathbb{N}}

\newcommand{\bbR}{\mathbb{R}}

\newcommand{\bbT}{\mathbb{T}}

\newcommand{\bbZ}{\mathbb{Z}}

\newcommand{\frakX}{\mathfrak{X}}

\newcommand{\fraka}{\mathfrak{a}}

\newcommand{\frakh}{\mathfrak{h}}

\newcommand{\frakl}{\mathfrak{l}}
\newcommand{\frakm}{\mathfrak{m}}

\newcommand{\rmd}{\mathrm{d}}

\renewcommand{\phi}{\varphi}

\newcommand{\eps}{\varepsilon}
\newcommand{\ldsb}{[\![}
\newcommand{\rdsb}{]\!]}
\newcommand{\ldab}{\langle\!\langle}
\newcommand{\rdab}{\rangle\!\rangle}
\renewcommand{\theta}{\vartheta}

\newcommand{\pr}{\operatorname{pr}}

\allowdisplaybreaks

\title[]{Deformations of Symplectic Foliations: algebraic aspects}

\author{Stephane Geudens}
\address{Department of Mathematics, University College London, 25 Gordon Street, London WC1H 0AY, UK}
\email{\href{mailto:s.geudens@ucl.ac.uk}{\underline{\smash{s.geudens@ucl.ac.uk}}}}

\author{Alfonso G. Tortorella}
\address{DipMat, Universit\`{a} degli Studi di Salerno, Via Giovanni Paolo II n°132, 84084  Fisciano (SA), Italy}
\email{\href{mailto:atortorella@unisa.it}{\underline{\smash{atortorella@unisa.it}}}}

\author{Marco Zambon}
\address{Department of Mathematics, KU Leuven, Celestijnenlaan 200B, 3001 Leuven, Belgium}
\email{\href{mailto:marco.zambon@kuleuven.be}{\underline{\smash{marco.zambon@kuleuven.be}}}}



\begin{document}
	
	\begin{abstract}
			In \cite{DefSF}, we developed the deformation theory of symplectic foliations, focusing on geometric aspects.
			Here, we address some algebraic questions that arose naturally. We show that the $L_{\infty}$-algebra constructed in \cite{DefSF} is independent of the choices made, and we prove that the gauge equivalence of Maurer-Cartan elements corresponds to the equivalence by isotopies of symplectic foliations.
	\end{abstract}
	
	\maketitle
	
	\setcounter{tocdepth}{1}
	
	\tableofcontents
	
	\section*{Introduction}

A symplectic foliation $(\mathcal{F},\omega)$ on a manifold $M$ is a foliation $\mathcal{F}$ equipped with a leafwise two-form $\omega\in\Omega^{2}(\mathcal{F})$ that is closed and non-degenerate. Equivalently, a symplectic foliation can be viewed as a regular Poisson structure, i.e. a constant rank bivector field $\Pi\in\mathfrak{X}^{2}(M)$ for which the Schouten-Nijenhuis bracket $[\Pi,\Pi]_{SN}$ vanishes. In a companion paper \cite{DefSF}, we developed the deformation theory of symplectic foliations, taking the Poisson geometry point of view. Our main result there describes small deformations of a given rank $2k$ symplectic foliation $(\mathcal{F},\omega)$ in terms of Maurer-Cartan elements of a suitable $L_{\infty}[1]$-algebra. More precisely, making a choice of distribution $G$ complementary to $T\mathcal{F}$, one can construct an $L_{\infty}[1]$-algebra $(\frakX_{\good}^\bullet(M)[2],\{\frakl^G_k\})$ and a bijection
\begin{equation}\label{eq:bijection}
\big\{\text{Small Maurer-Cartan elements of}\ (\frakX_{\good}^\bullet(M)[2],\{\frakl^G_k\})\big\}\leftrightarrow\big\{(\widetilde{\mathcal{F}},\widetilde{\omega})\in\SymplFol^{2k}(M):\ T\widetilde{\mathcal{F}}\pitchfork G\big\}.
\end{equation}

\smallskip
The present paper is a follow-up of \cite{DefSF} which answers two mostly algebraic questions that were left unaddressed in \cite{DefSF}. The problems that we consider here are the following.

First, while the construction of the $L_{\infty}[1]$-algebra $(\frakX_{\good}^\bullet(M)[2],\{\frakl^G_k\})$ involves a choice of complement $G$ to the foliation $T\mathcal{F}$, we show in Cor.~\ref{cor:preservessubalg} that different choices of complement produce $L_{\infty}[1]$-algebras that are canonically isomorphic. Hence, any symplectic foliation $(\mathcal{F},\omega)$ {has an attached} $L_{\infty}[1]$-algebra governing its deformation problem, which is essentially canonical. 

Second, we address equivalences of deformations of a symplectic foliation. On one hand, there is a \emph{geometric} notion of equivalence, which declares two symplectic foliations to be equivalent if they are related by a diffeomorphism isotopic to the identity. On the other hand, there is the \emph{algebraic} notion of gauge equivalence for Maurer-Cartan elements of the $L_{\infty}[1]$-algebra $(\frakX_{\good}^\bullet(M)[2],\{\frakl^G_k\})$. We show in Thm.~\ref{thm:gauge} that these two notions of equivalence essentially agree under the bijection \eqref{eq:bijection}. This result can be used to draw conclusions about the moduli space of symplectic foliations $\SymplFol^{2k}(M)/\text{Diff}_{0}(M)$, where we quotient by isotopies. It shows in particular that the cohomology group $H^{2}_{\mathcal{F}}(M,\Pi)$ of the $L_{\infty}[1]$-algebra $(\frakX_{\good}^\bullet(M)[2],\{\frakl^G_k\})$ plays the role of formal tangent space to the moduli space, i.e.
\[
T_{[(\mathcal{F},\omega)]}\big(\SymplFol^{2k}(M)/\text{Diff}_{0}(M)\big)\cong H^{2}_{\mathcal{F}}(M,\Pi).
\]
We show in Ex.~\ref{ex:kronecker} that this formal tangent space can vary drastically from point to point, which prevents the moduli space $\SymplFol^{2k}(M)/\text{Diff}_{0}(M)$ from being smooth.

\bigskip
\noindent
\textbf{Acknowledgements:}	
S.G. acknowledges support from the UCL Institute for Mathematics \& Statistical Science (IMSS). 
During the preparation of this paper, A.G.T.~has been supported by the FWO postdoctoral fellowship 1204019N at KU Leuven, the FCT postdoctoral fellowship 2021.04201.CEECIND at CMUC, and also by CMUP which is financed by national funds through FCT – Fundação para a Ciência e a Tecnologia, I.P., under the project with reference UIDB/00144/2020.
Further he is member of the National Group for Algebraic and Geometric Structures, and their Applications (GNSAGA – INdAM).
M.Z. acknowledges partial support by
the FWO and FNRS under EOS projects G0H4518N and G0I2222N, by FWO project G0B3523N, and by the long term structural funding -- Methusalem grant of the Flemish Government (Belgium).

	\section{Background on Symplectic Foliations and their Deformations}\label{sec:one}
	
	We recall the definitions and statements of \cite{DefSF} that will be needed in the sequel. The objects of interest are symplectic foliations, and our main result in \cite{DefSF} is a local parametrization for the space of symplectic foliations in terms of Maurer-Cartan elements of a suitable $L_{\infty}[1]$-algebra.
{We refer to Appendix \ref{app:L_infty_algebras} for the necessary background on $L_{\infty}[1]$-algebras and to Appendix \ref{app:Diracgeom} for  Dirac geometry.}

\subsection{Symplectic foliations, regular Poisson structures and good multivector fields}	
We introduce symplectic foliations and review how they can be seen as regular Poisson structures. We also recall the space of good multivector fields, as introduced in  \cite{DefSF}, which supports the $L_{\infty}[1]$-structure that we aim to describe in this section. This subsection is based on  \cite[\S 1]{DefSF}.	

	\begin{definition}
		\label{def:symplectic_foliation}
		A \emph{symplectic foliation} on a manifold $M$ consists of a foliation $\calF$ of $M$ with a \emph{leafwise symplectic structure} $\omega$, i.e.~a non-degenerate $2$-cocycle $\omega$ in the leafwise de Rham complex $(\Omega^\bullet(\calF),\rmd_\calF)$.
	\end{definition}
	
	We denote by $\SymplFol^{2k}(M)$ the space of all symplectic foliations on $M$ having rank $2k$. We studied this space in \cite{DefSF} taking a slightly different point of view. Namely, we took advantage of the fact that a symplectic foliation is the same thing as a regular Poisson structure, as we now explain.

	\begin{definition}
		A \emph{Poisson structure} on a manifold $M$ is a bivector field $\Pi\in\frakX^2(M)$ for which the Schouten-Nijenhuis bracket $[\Pi,\Pi]_{SN}$ vanishes. A Poisson structure $\Pi$ is \emph{regular} if the associated vector bundle morphism $\Pi^\sharp:T^\ast M\to TM,\ \eta\mapsto\iota_\eta\Pi$ has constant rank.
		In this case, we refer to the rank of $\Pi^\sharp$ as the \emph{rank} of $\Pi$. The distribution $T\calF:=\im\Pi^\sharp$ integrates to the \emph{characteristic foliation} $\calF$ of $\Pi$.
	\end{definition}
	
 	We denote by $\RegPoiss^{2k}(M)$ the space of all regular Poisson structures on $M$ having rank $2k$. There is a bijection between $\RegPoiss^{2k}(M)$ and $\SymplFol^{2k}(M)$, because of the following.
	Note that a regular bivector field $\Pi$ with $\im\Pi^\sharp=T\calF$ can be viewed as a non-degenerate element of $\frakX^{2}(\calF):=\Gamma(\wedge^{2}T\calF)$.
	Therefore, a regular Poisson structure with characteristic foliation $\calF$ is the same thing as a non-degenerate Maurer-Cartan element of the graded Lie algebra $(\frakX^\bullet(\calF)[1],[-,-]_{\sf SN})$.
	Further, the relation $\Pi=-\omega^{-1}$ establishes a one-to-one correspondence between non-degenerate Maurer-Cartan elements $\Pi$ of $(\frakX^\bullet(\calF)[1],[-,-]_{\sf SN})$ and non-degenerate $2$-cocycles $\omega$ in the leafwise de Rham complex $(\Omega^\bullet(\calF),\rmd_\calF)$.
	This yields the correspondence
	\begin{equation*}
			\SymplFol^{2k}(M)\longrightarrow\RegPoiss^{2k}(M),\ \ (\calF,\omega)\longmapsto -\omega^{-1}.
		\end{equation*}

	In \cite{DefSF}, we gave a parametrization of the space $\RegPoiss^{2k}(M)$ around a given regular Poisson structure $\Pi$. It involves a suitable subspace of the multivector fields on $M$, which we call \emph{good multivector fields}.

	\begin{definition}
		\label{def:good_multi-vector_fields}
		Let $(M,\Pi)$ be a regular Poisson manifold with {characteristic} foliation $\calF$, and denote by $T\calF^{0}$ the annihilator of $T\calF$. The graded space of \emph{good multivector fields} is defined by
		\begin{equation*}
			\label{eq:good_k-vectorfields}
			\frakX^\bullet_\good(M)=\{W\in\frakX^\bullet(M):\iota_\alpha\iota_\beta W=0\ \text{for all}\ \alpha,\beta\in\Gamma(T^\circ\calF)\}.
		\end{equation*}
	\end{definition}

Note that $\frakX^\bullet_\good(M)$ is the graded $\frakX^\bullet(\calF)$-submodule of $\frakX^\bullet(M)$ generated by $C^\infty(M)\oplus\frakX(M)$. {We denote by $d_{\Pi}:=[\Pi,-]_{SN}$ the Poisson differential.}
It turns out that $(\frakX^\bullet_\good(M),d_{\Pi})$ is a subcomplex of the Lichnerowicz complex $(\frakX^{\bullet}(M),d_{\Pi})$, whose $2$-cocycles are infinitesimal deformations of $\Pi$ {as a regular Poisson structure}:
\begin{equation*}
	T_\Pi\big(\RegPoiss^{2k}(M)\big)=\ker\big(\rmd_\Pi:\frakX^2_\good(M)\to\frakX^3_\good(M)\big).
\end{equation*}

	\subsection{Parametrizing Nearby Regular Bivector Fields}
	\label{subsec:param}
	
This subsection is based on  \cite[\S 2]{DefSF}.	
	Given a rank $2k$ symplectic foliation $(\calF,\omega)$ on $M$, let $\Pi$ be the corresponding rank $2k$ regular Poisson structure, so that $T\calF=\im\Pi^\sharp$. From now on, we fix the following pieces of data:
	
\bigskip	
	\fbox{\parbox{0.95\textwidth}{
	\begin{itemize}
	    \item $G$ is a distribution complementary to $T\calF$, i.e.~$TM=G\oplus T\calF$.
	    \item $\gamma$ is the extension of the leafwise symplectic form $\omega$ by zero on $G$, i.e. it satisfies
	    \begin{equation*}
		\gamma^\flat|_{T\calF}=\omega^\flat\qquad\text{and}\qquad\gamma^\flat|_G=0.
	\end{equation*}
	\item We denote by $\pr_{T\calF}:TM\rightarrow T\calF$ and $\pr_G: TM\rightarrow G$ the projection maps.
	\end{itemize}
	}}
\smallskip		
	
Here $\gamma^\flat\colon TM\to T^*M$ is defined by $\gamma^\flat(X)=\iota_X\gamma$.
	We will parametrize the bivector fields close to $\Pi$ using the so-called \emph{Dirac exponential map} associated with $G$ and $\Pi$. 
	This parametrization turns out to be well-suited to single out the bivector fields  close to $\Pi$ which are additionally regular of rank $2k$.
	
\bigskip

Given a bivector field $Z\in\frakX^{2}(M)$, the two-form $\gamma\in\Omega^{2}(M)$ can be used to produce a new bivector field $Z^{\gamma}\in\frakX^{2}(M)$, provided that $Z$ is small enough. More precisely, $Z$ is required to belong to the  neighborhood $\calI_\gamma$ of the zero section in $\wedge^2TM$ given by
\begin{equation*}
		\label{eq:def:I_gamma}
		\calI_\gamma:=\sqcup_{x\in M}\{Z_x\in\wedge^2T_xM\mid \id_{T^\ast_xM}+\gamma^\flat_x\circ Z_x^\sharp:T_x^\ast M\to T_x^\ast M\ \text{is invertible}\}.
\end{equation*}
The bivector field $Z^{\gamma}$ in question, called the \emph{gauge transform of $Z$ by $\gamma$}, is determined by
\begin{equation}\label{eq:gammasharp}
(Z^{\gamma})^{\sharp}=Z^{\sharp}\circ(\id_{T^{*}M}+\gamma^{\flat}\circ Z^{\sharp})^{-1}.
\end{equation}
We can now define the Dirac exponential map, which parametrizes bivector fields close to $\Pi$.

	\begin{definition}
		The \emph{Dirac exponential map} $\exp_G$ associated with $G$ and $\Pi$ is the map
		\[
		\exp_G\colon\calI_\gamma\xrightarrow{\sim}\Pi+\calI_{-\gamma}, \ \ Z\mapsto\Pi+Z^\gamma.
		\]
		It is characterized by the property that for all $Z\in\Gamma(\calI_\gamma)$,
		\[(\exp_G(Z))^\sharp=\Pi^\sharp+Z^\sharp\circ(\id_{T^\ast M}+\gamma^\flat\circ Z^\sharp)^{-1}.
		\]
	\end{definition}
	
	\begin{remark}\label{rem:exp-action}
	The map $\exp_G$ has a natural description in Dirac geometric terms. Recall that the two-form $\gamma\in\Omega^{2}(M)$ and the bivector field $\Pi\in\frakX^{2}(M)$ determine orthogonal transformations $\calR_{\gamma}, \calR_{\Pi}$ of the generalized tangent bundle $(\bbT M,\ldab-,-\rdab)$, given by
	\begin{align*}
	&\calR_\gamma:\bbT M\to\bbT M:\ X+\alpha\mapsto X+\alpha+\iota_X\gamma,\\
	&\calR_\Pi:\bbT M\to\bbT M: X+\alpha\mapsto X+\Pi^{\sharp}(\alpha)+\alpha.
	\end{align*}
	The action of $\exp_G$ on sections $Z\in\Gamma(\calI_\gamma)$ is then given by
\begin{equation}\label{eq:grexp}
  \gr(\exp_G(Z))=\calR_\Pi\calR_\gamma\gr(Z),
\end{equation}	
{where $\gr$ denotes the graph.} 
	\end{remark}

	Parametrizing bivector fields close to $\Pi$ with the Dirac exponential map has the advantage that the constant rank condition is turned into a linear condition. Namely, the parameter is required to belong to the linear subspace of good bivector fields (see Def.~\ref{def:good_multi-vector_fields}). Consequently, upon restricting $\exp_G$ to $\frakX_\good^2(M)$, we obtain a submanifold chart for the space of regular rank $2k$ bivector fields on $M$.

	\begin{theorem}
		\label{theor:nearby_regular_bivectors}
		The Dirac exponential map
		\begin{equation*}
			\exp_G\colon\calI_{\gamma}\overset{\sim}{\longrightarrow}\Pi+\calI_{-\gamma},\ \  Z\longmapsto\exp_G(Z):=\Pi+Z^\gamma
		\end{equation*}
		induces the following bijection at the level of sections:
		\begin{equation*}
			\Gamma(\calI_{\gamma})\cap\frakX_\good^2(M)\overset{\sim}{\longrightarrow}\{W\in\frakX^2_{\text{reg-2k}}(M)\mid \im W^\sharp\pitchfork G\},\ \ Z\longmapsto\exp_G(Z).
		\end{equation*}
	\end{theorem}

	\subsection{Parametrizing nearby regular Poisson structures}
	\label{sec:Koszul_algebra}
	
	By Thm.~\ref{theor:nearby_regular_bivectors}, a local parametrization for the space of rank $2k$ regular Poisson structures is obtained if we can single out the elements in $\Gamma(\calI_{\gamma})\cap\frakX_\good^2(M)$ whose image under the Dirac exponential map is a Poisson structure. We will see that these elements are exactly the Maurer-Cartan elements in $\Gamma(\calI_{\gamma})$ for a suitable $L_{\infty}[1]$-algebra structure on the shifted space $\frakX_\good^{\bullet}(M)[2]$ of good multivector fields. This subsection is based on  \cite[\S~3.1, \S~3.2]{DefSF}.	

	\bigskip
	
{Recall that a bivector field is Poisson if{f} its graph is a Dirac structure.}	We saw in Rem.~\ref{rem:exp-action} that the map $\exp_G$ has a natural description in terms of Dirac geometry. This leads us to studying deformations of the Dirac structure $\gr\Pi\subset \bbT M$, rather than studying deformations of $\Pi$ directly.
	Deformations of a given Dirac structure $A\subset \bbT M$ are described conveniently by choosing a complementary almost Dirac structure $B$, as recalled in Appendix \ref{app:Deformations_Dirac_Structures}. Such a choice yields an $L_{\infty}$-algebra $(\Omega^\bullet(A)[1],\{\frakm_k^{B}\})$ whose {Maurer-Cartan} (MC) elements parametrize the Dirac structures transverse to $B$. For the Dirac structure $A:=\gr\Pi$, the most straightforward choice of complement would be $B:=TM$, but this approach is not well-suited to single out \emph{regular} deformations of $\Pi$.  Indeed, this way one recovers (up to isomorphism) the usual Koszul dgLa $(\frakX^\bullet(M)[1],\rmd_\Pi, [-,-]_{SN})$, whose MC elements $Z$ parametrize Poisson structures $P\in\mathfrak{X}^{2}(M)$ via the relation
 \[
 P=\Pi+Z.
 \]
 However, the bivector fields $Z\in\mathfrak{X}^{2}(M)$ for which $\Pi+Z$ has constant rank $2k$ do not form a vector space, which is why the complement $TM$ is not suitable for our purposes. A better choice is the following.

\begin{lemma}
The almost Dirac structure $G\oplus G^{0}\subset\bbT M$ is complementary to $\gr\Pi$.
\end{lemma}	

It would be natural to apply Lemmas~\ref{lem:2.6first} and~\ref{lem:2.6second} with input $E=\bbT M$, $A=\gr\Pi$ and $B=G\oplus G^{0}$, yielding an $L_\infty[1]$-algebra $\big(\Omega^\bullet(\gr\Pi)[2],\big\{\mu^{G\oplus G^{0}}_k\big\}\big)$ which controls the deformation problem of the Dirac structure $\gr\Pi$. However, we take an indirect but more convenient approach instead.
Since we want to parametrize deformations of $\Pi$ with the Dirac exponential map, whose action on sections is given by $\calR_{\Pi}\calR_{\gamma}$ (see Rem.~\ref{rem:exp-action}), we first transport the standard Courant algebroid structure along $\calR_{\Pi}\calR_{\gamma}$. Doing so also simplifies the situation, since it transforms the splitting $\bbT M=\gr\Pi\oplus(G\oplus G^{0})$ into $\bbT M=T^{*}M\oplus TM$. Only then we apply Lemma~\ref{lem:2.6first}, yielding an $L_\infty[1]$-algebra structure $\{\frakl_k^{G}\}$ on $\frakX^\bullet(M)[2]$.

	\begin{lemma}
	\label{lem:Koszul_algebra:Courant_iso}
	There exists a unique Courant algebroid structure $(\ldsb-,-\rdsb_G,\rho_G)$ on $(\bbT M,\ldab-,-\rdab)$, 
	such that the orthogonal transformation
	\begin{align*}
	&\calR_\Pi\calR_\gamma:(\bbT M,\ldab-,-\rdab,\ldsb-,-\rdsb_G,\rho_G)\overset{\sim}{\longrightarrow}(\bbT M,\ldab-,-\rdab,\ldsb-,-\rdsb,\pr_{TM}),\nonumber\\
	&\hspace{4.05cm}X+\alpha\longmapsto (\pr_GX+\iota_\alpha\Pi)+(\alpha+\iota_X\gamma),
	\end{align*}
	is a Courant algebroid isomorphism.
	In particular, the latter induces: 
\begin{itemize}
\item 
	the Lie algebroid isomorphism
	\begin{equation}
	\label{eq:lem:Koszul_algebra:Lie_algbd_iso}
	\calR_\Pi|_{T^\ast M}:T^\ast M\overset{\sim}{\longrightarrow}\gr\Pi,\ \alpha\longmapsto\iota_\alpha\Pi+\alpha,
	\end{equation}
	where $T^\ast M$ carries the Lie algebroid structure associated with $\Pi$.
\item  the almost Lie algebroid isomorphism
	\begin{equation}
	\label{eq:lem:Koszul_algebra:almost_Lie_algbd_iso}
	(\calR_\Pi\calR_\gamma)|_{TM}:TM\overset{\sim}{\longrightarrow}G\oplus G^{0},\ X\longmapsto\pr_GX+\iota_X\gamma,
	\end{equation}
	where $TM$ carries the almost Lie algebroid structure $([-,-]_\gamma,\rho_\gamma)$ defined by
	\begin{equation}
	\label{eq:lem:Koszul_algebra:almost_Lie_algbd_structure}
	[X,Y]_\gamma=[\pr_GX,\pr_GY]-\Pi^\sharp(\calL_{\pr_GX}\iota_Y\gamma-\calL_{\pr_GY}\iota_X\gamma),\qquad\rho_\gamma(X)=\pr_GX.
	\end{equation}
	\end{itemize}
\end{lemma}
	
\begin{remark}\label{rem:Courant_tensor_twisted_TM}
  Denote by $\Upsilon_{TM}^G\in\Omega^3(M)$  the Courant tensor of the almost Dirac structure $TM$ in $(\bbT M,\ldab-,-\rdab,\ldsb-,-\rdsb_G,\rho_G)$, as in Rem.~\ref{rem:Courant_tensor}. It follows from Lemma~\ref{lem:Koszul_algebra:Courant_iso} that $\Upsilon_{TM}^G$ is the pullback along the isomorphism~\eqref{eq:lem:Koszul_algebra:almost_Lie_algbd_iso} of the Courant tensor of $G\oplus G^{0}$.
Consequently, for all $X,Y,Z\in\frakX(M)$,
	\begin{equation*}
	\Upsilon^G_{TM}(X,Y,Z)=\gamma(X,[\pr_G Y,\pr_G Z])+\gamma(Y,[\pr_G Z,\pr_G X])+\gamma(Z,[\pr_G X,\pr_G Y]).
	\end{equation*}	
Notice in particular that $\Upsilon^G_{TM}$ vanishes exactly when $G\subset TM$ is involutive.
\end{remark}
	
We now apply Lemma~\ref{lem:2.6first}  to the case where $E=(\bbT M,\ldab-,-\rdab,\ldsb-,-\rdsb_G,\rho_G)$ is the Courant algebroid mentioned in Lemma~\ref{lem:Koszul_algebra:Courant_iso}, with Dirac structure $A=T^\ast M$ and almost Dirac structure $B=TM$.
So the graded space $\frakX^\bullet(M)[2]$ inherits an $L_\infty[1]$-algebra structure $\{\frakl_k^G\}$, described in the following.

	\begin{proposition}
	\label{prop:Koszul_algebra}
	The $L_\infty[1]$-algebra associated with the Dirac structure $T^\ast M$ via the complementary almost Dirac structure $TM$ in the Courant algebroid $(\bbT M,\ldab-,-\rdab,\ldsb-,-\rdsb_G,\rho_G)$ 
	is $(\frakX^\bullet(M)[2],\{\frakl_k^G\})$, whose only non-trivial multibrackets $\frakl_1^G,\frakl_2^G,\frakl_3^G$ are given by the following.
	\begin{itemize}
		\item The unary bracket $\frakl_1^G$ is the Poisson differential $\rmd_\Pi$, i.e. for all $P\in\frakX^\bullet(M)$,
		\begin{equation*}
		\frakl_1^G(P)=[\Pi,P]_{\sf SN},
		\end{equation*}
		\item The binary bracket $\frakl_2^G$ acts as follows on homogeneous $P,Q\in\frakX^\bullet(M)$:
		\begin{equation*}
		\frakl_2^G(P,Q)=(-)^{|P|} [P,Q]_\gamma,
		\end{equation*}
	where $[-,-]_\gamma$ is the
	extension to an	almost Gerstenhaber bracket 
	of the bracket in eq. \eqref{eq:lem:Koszul_algebra:almost_Lie_algbd_structure}.
		\item The ternary bracket $\frakl_3^G$ acts as follows on  homogeneous $P,Q,R\in\frakX^\bullet(M)$:
		\begin{equation*}
		\frakl_3^G(P,Q,R)={-}(-1)^{|Q|}\big(P^\sharp\wedge Q^\sharp \wedge R^\sharp\big)\Upsilon^G_{TM},
		\end{equation*}
		{where $\Upsilon^G_{TM}$  is defined in Rem.~\ref{rem:Courant_tensor_twisted_TM}.}
	\end{itemize}
\end{proposition}
By Rem.~\ref{rem:Courant_tensor_twisted_TM}, the ternary bracket $\frakl_3^G$ vanishes exactly when $G$ is involutive, in which case the above $L_{\infty}[1]$-algebra reduces to a dgL[1]a.  
Also note that, as a consequence of Lemma~\ref{lem:Koszul_algebra:Courant_iso}, the $L_\infty[1]$-algebra $(\frakX^\bullet(M)[2],\{\frakl_k^G\})$ is canonically isomorphic to the $L_\infty[1]$-algebra $\big(\Omega^\bullet(\gr\Pi)[2],\big\{\mu^{G\oplus G^{0}}_k\big\}\big)$ defined in terms of the splitting $\bbT M=\gr\Pi\oplus(G\oplus G^{0})$ of the standard Courant algebroid.
	
	\begin{proposition}
		\label{prop:Koszul_algebra:isomorphism}
		The bundle isomorphism~\eqref{eq:lem:Koszul_algebra:Lie_algbd_iso} induces a strict isomorphism of $L_\infty[1]$-algebras
		\begin{equation*}
			\wedge^\bullet(\calR_\Pi|_{T^\ast M})^\ast:\big(\Omega^\bullet(\gr\Pi)[2],\big\{\mu^{G\oplus G^{0}}_k\big\}\big)\overset{\sim}{\longrightarrow}(\frakX^\bullet(M)[2],\{\frakl_k^G\}).
		\end{equation*}
	\end{proposition}

	The MC elements $Z$ of $(\frakX^\bullet(M)[2],\{\frakl^G_k\})$
    encode the Dirac structures $L\subset\bbT M$ transverse to $G\oplus G^{0}$ via the relation $L=\calR_\Pi\calR_\gamma\gr(Z)$.
{If $Z\in\Gamma(\calI_{\gamma})$, then} {we can use \eqref{eq:grexp} to express this relation as
 \begin{equation}\label{eq:assignMC}
			L= \gr(\exp_G(Z)).
		\end{equation}
}
Thm.~\ref{theor:nearby_regular_bivectors} shows that, restricting this correspondence to MC elements $Z\in\Gamma(\calI_{\gamma})\cap\frakX_\good^2(M)$, one obtains a parametrization for the space of regular Poisson structures near $\Pi$.
	Hence, if we show that $\frakX_\good^\bullet(M)[2]$ is an $L_{\infty}[1]$-subalgebra of $(\frakX^\bullet(M)[2],\{\frakl_k^G\})$, then we find an $L_\infty[1]$-algebra for which small MC elements parametrize deformations of the symplectic foliation $(\calF,\omega)$.

	\begin{proposition}
		\label{prop:good_L_infty_algebra}
		The multibrackets $\frakl_k^G$ map $\frakX_\good^\bullet(M)[2]\subset\frakX^\bullet(M)[2]$ to itself.
		So $\frakX_\good^\bullet(M)[2]$ inherits an $L_\infty[1]$-algebra structure, still denoted by $\{\frakl_k^G\}$.
	\end{proposition}

	\begin{theorem}{\bf (Main theorem of \cite{DefSF})}
	\label{theor:deformation_theory regular_Poisson}
	Let $(\calF,\omega)$ be a rank $2k$ symplectic foliation on $M$, with corresponding regular Poisson structure $\Pi$.
	Fix a distribution $G$ complementary to $T\calF$.
	The relation 
	\begin{equation*}
	{-}\widetilde\omega^{-1}=\exp_G(Z)
	\end{equation*}
	establishes a canonical one-to-one correspondence between
	\begin{enumerate}
		\item MC elements $Z$ of the $L_\infty[1]$-algebra $(\frakX_\good^\bullet(M)[2],\{\frakl^G_k\})$ 
		such that $Z\in \Gamma(\calI_{\gamma})$,
			\item rank $2k$ symplectic foliations $(\widetilde\calF,\widetilde\omega)$ on $M$ 
		such that $T\widetilde\calF\pitchfork G$. 
	\end{enumerate}
\end{theorem}

\section{Change of Splitting}\label{sec:two}	

The construction of the $L_{\infty}[1]$-algebra $(\frakX_{\calF}^\bullet(M)[2],\{\frakl^G_k\})$ of a symplectic foliation $(\calF,\omega)$ involves a choice of distribution $G$ complementary to $T\calF$. In this section, we show that different complements $G_0,G_1$ produce $L_{\infty}[1]$-algebras that are canonically isomorphic, see Cor.~\ref{cor:preservessubalg}.
 We proceed in two steps.

First, recall {from \S \ref{sec:Koszul_algebra}}  that the $L_{\infty}[1]$-algebras 
$(\frakX^\bullet(M)[2],\{\frakl^{G_0}_k\})$ and $(\frakX^\bullet(M)[2],\{\frakl^{G_1}_k\})$ are essentially obtained by deforming the Dirac structure $\gr\Pi$, using $G_0\oplus G_0^{0}$ and $G_1\oplus G_1^{0}$ as almost Dirac complements in Lemma \ref{lem:2.6first}.  It was shown in \cite{GMS} that different choices of almost Dirac complement  yield isomorphic $L_{\infty}[1]$-algebras. Another proof of this fact appeared in \cite{Jacobi}, in the more general setting of Dirac-Jacobi structures. We specialize the latter proof to the Dirac setting, thus obtaining an explicit $L_{\infty}[1]$-isomorphism between 
$(\frakX^\bullet(M)[2],\{\frakl^{G_0}_k\})$ and $(\frakX^\bullet(M)[2],\{\frakl^{G_1}_k\})$.
Second, we check that this isomorphism restricts to an isomorphism between the $L_{\infty}[1]$-subalgebras  $(\frakX_{\calF}^\bullet(M)[2],\{\frakl^{G_0}_k\})$ and $(\frakX_{\calF}^\bullet(M)[2],\{\frakl^{G_1}_k\})$.

\subsection{The Dirac case}	\label{subsec:Diraccase}

{Here we carry out the first step above for arbitrary Dirac structures.}
Given a Dirac structure $A$ in a Courant algebroid $E$, a choice of complementary almost Dirac structure $B$ yields an $L_{\infty}[1]$-algebra $(\Omega^{\bullet}(A)[2],\{\mu_{k}^{B}\})$ (see Lemma \ref{lem:2.6first}). {In the first part of this subsection} we recall how one proves in \cite{Jacobi} that this $L_{\infty}[1]$-algebra is independent of the choice of complement $B$. The proof takes advantage of the fact that the $L_{\infty}[1]$-algebra in question arises from Voronov's derived bracket construction \cite{voronov2005higher}, as explained in Appendix~\ref{app:Deformations_Dirac_Structures}. This allows one to use a result from \cite{CS} on equivalences of higher derived brackets. For more details, we refer to \cite[\S 5.2]{Jacobi}.

The setup is as follows. Let $B_0,B_1$ be almost Dirac complements to the Dirac structure $A\subset E$. 
As described in Appendix~\ref{app:Deformations_Dirac_Structures}, via $E=A\oplus B_i\cong A\oplus A^*$, one obtains Maurer-Cartan elements $\Theta_0,\Theta_1$ of the graded Lie algebra $(C^\infty(T^\ast[2]A[1])[2],\{-,-\})$, where $\{-,-\}$ denotes the canonical degree $-2$ Poisson structure on {the graded manifold} $T^\ast[2]A[1]$. These give rise to two sets of V-data, for $i=0,1$:
\begin{equation}\label{eq:Vdata}
\big(C^\infty(T^\ast[2]A[1])[2],\Omega^\bullet(A)[2],P,\Theta_i\big), 
\end{equation}
where $P:C^\infty(T^\ast[2]A[1])[2]\rightarrow\Omega^\bullet(A)[2]$ is the restriction to the zero section. Applying Voronov's construction, we obtain two codifferentials $\calQ_0,\calQ_1$ on the symmetric coalgebra ${\sf S}(\Omega^\bullet(A)[2])$. Upon changing the sign of the resulting binary brackets (c.f. Rem.~\ref{rem:decalage}), we end up with two $L_{\infty}[1]$-algebras $(\Omega^{\bullet}(A)[2],\{\mu_{k}^{0}\})$ and $(\Omega^{\bullet}(A)[2],\{\mu_{k}^{1}\})$, as described in Lemma \ref{lem:2.6first}.

		In order to relate these $L_{\infty}[1]$-algebras, we start by expressing the almost Dirac structure $B_1$ in the direct sum $E=A\oplus B_0$. As $B_1$ is transverse to the first summand $A$, it can be realized as the graph of a vector bundle map $B_0\cong A^*\to A$, hence	 there exists $\eta\in\Gamma(\wedge^2A)$ such that
		\begin{equation}\label{eq:eta}
			B_1=\gr(\eta)\equiv\{\iota_\alpha\eta+\alpha\mid\alpha\in A^\ast\}\subset A\oplus A^\ast\simeq A\oplus B_0=E.
		\end{equation}
		Since $\eta$ is a degree $2$ function in $C^\infty(T^\ast[2]A[1])$, it determines a degree zero graded Lie algebra derivation of $(C^\infty(T^\ast[2]A[1])[2],\{-,-\})$, namely
		\[
		{\sf m}:=\{\eta,-\}:C^\infty(T^\ast[2]A[1])[2]\to C^\infty(T^\ast[2]A[1])[2].
		\]
		To apply the equivalence result from \cite{CS}, one checks that the following technical assumptions are met:
		\begin{enumerate}
		    \item The equality $P\circ{\sf m}\circ P=P\circ {\sf m}$ holds.
		    \item The zero solution $\lambda_\epsilon=0$ is the only solution to the Cauchy problem
		    \[
		    \begin{cases}
		    \frac{\rmd}{\rmd\varepsilon}\lambda_\varepsilon=P{\sf m}\lambda_\varepsilon\\ 
		    \lambda_0=0
		    \end{cases}.
		    \]
		\end{enumerate}
		In fact, $P\circ {\sf m}=0$, {as we show in the proof of Lemma \ref{lem:M2} below,} so these requirements are trivially satisfied. Since the derivation ${\sf m}$ is pronilpotent, it has a well-defined flow, given by the 1-parameter group of graded Lie algebra automorphisms $e^{\eps{\sf m}}$.

		Applying the flow $e^{\eps{\sf m}}$ to the Maurer-Cartan element $\Theta_0$ of  $(C^\infty(T^\ast[2]A[1])[2],\{-,-\})$ yields a family of Maurer-Cartan elements
		$
		\Theta(\eps):=e^{\eps{\sf m}}\Theta_0
		$
		that stays inside $\ker P$, so we obtain a family of V-data 
		\[
		\big(C^\infty(T^\ast[2]A[1])[2],\Omega^\bullet(A)[2],P,\Theta(\eps)\big).
		\]
		One checks that this family interpolates between the V-data $\big(C^\infty(T^\ast[2]A[1])[2],\Omega^\bullet(A)[2],P,\Theta_0\big)$ and $\big(C^\infty(T^\ast[2]A[1])[2],\Omega^\bullet(A)[2],P,\Theta_1\big)$ that we started with (c.f. \eqref{eq:Vdata}). Consequently, Voronov's method gives rise to a family of codifferentials $\calQ(\eps)$ on ${\sf S}(\Omega^\bullet(A)[2])$ interpolating between $\calQ_0$ and $\calQ_1$.

		Following further the strategy from \cite{CS}, we define ${\sf M}$ to be the degree $0$ coderivation of ${\sf S}(\Omega^\bullet(A)[2])$ with Taylor coefficients $\{{\sf M}_k\}$ given by
		\begin{equation}
			\label{eq:M_k}
			{\sf M}_k(\omega_1\odot\ldots\odot\omega_k)=P\{\{\ldots\{{\sf m}\omega_1,\omega_2\},\ldots\},\omega_k\}
		\end{equation}
		for all $\omega_1,\dots,\omega_k\in\Omega^\bullet(A)[2]$.
We recall that by definition, the $k$-th Taylor coefficient of ${\sf M}$ is given by  $M_k=pr_{\Omega^\bullet(A)[2]}\circ M|_{S^k(\Omega^\bullet(A)[2])}$, see Prop.~\ref{prop:coalgebra_coderivation}.		
		 By \cite[Prop.~3.4]{CS}, the coderivation ${\sf M}$ exponentiates to a $1$-parameter group of graded coalgebra automorphisms $e^{\eps{\sf M}}$ of ${\sf S}(\Omega^\bullet(A)[2])$.  Moreover, by \cite[Thm.~3.2]{CS}, the coalgebra automorphism $e^{\eps{\sf M}}$ intertwines the codifferentials $\calQ_0$ and $\calQ(\eps)$. At time $\eps=1$, we obtain the following result, which recovers in the case of exact Courant algebroids  \cite[Cor.~3.9]{GMS}.

		\begin{theorem}
			\label{theor:GMS}
			The coalgebra automorphism $e^{\sf M}$ corresponds to an $L_\infty[1]$-algebra isomorphism
			\begin{equation*}
				e^{\sf M}:(\Omega^{\bullet}(A)[2],\mu_{1}^{0},-\mu_{2}^{0},\mu_{3}^{0})\longrightarrow(\Omega^{\bullet}(A)[2],\mu_{1}^{1},-\mu_{2}^{1},\mu_{3}^{1}).
			\end{equation*}
Precomposing and postcomposing with the transposition
\[
\tau:\Omega^{\bullet}(A)[2]\rightarrow\Omega^{\bullet}(A)[2]:\alpha\mapsto -\alpha,
\]
we obtain a canonical $L_{\infty}[1]$-isomorphism
\[
\tau\circ e^{\sf M}\circ\tau:(\Omega^{\bullet}(A)[2],\mu_{1}^{0},\mu_{2}^{0},\mu_{3}^{0})\longrightarrow(\Omega^{\bullet}(A)[2],\mu_{1}^{1},\mu_{2}^{1},\mu_{3}^{1}).
\]
\end{theorem}	
	
We now give an explicit description of the coderivation ${\sf M}$.
We include a detailed proof because the result will be important for Prop.~\ref{prop:subalgebras} later, and it does not appear in \cite{Jacobi} or \cite{GMS} in this generality.

\begin{lemma}\label{lem:M2}
Let $A\rightarrow M$ be any vector bundle. Consider $T^{*}[2]A[1]$ with its standard {degree $-2$} Poisson structure $\{-,-\}$, and denote by $P:C^{\infty}(T^{*}[2]A[1])\rightarrow\Omega^{\bullet}(A)$ the restriction to the zero section. Let $\eta\in\Gamma(\wedge^{2}A)$ be arbitrary and put ${\sf m}:=\{\eta,-\}$. We define $\sf M$ to be the coderivation of ${\sf S}(\Omega^\bullet(A))$ with Taylor coefficients
\[
\sf M_k(\omega_1\odot\cdots\odot\omega_k)=P\{\{\ldots\{{\sf m}\omega_1,\omega_2\},\ldots\},\omega_k\}.
\]
Then all coefficients $\sf M_k$ are zero except for $\sf M_2$, whose action on homogeneous $\omega_1,\omega_2\in\Omega^{\bullet}(A)$ reads
\begin{equation}\label{eq:M2bis}
{\sf M}_2(\omega_1\odot\omega_2)=(-1)^{|\omega_1|}(\omega_{1}^{\flat}\wedge\omega_2^{\flat})\eta.
\end{equation}
\end{lemma}

\begin{remark}\label{rem:rewrite}
One can check that the formula \eqref{eq:M2bis} specializes as follows on two-forms:
\[
({\sf M}_2(\omega_1\odot\omega_2))^{\flat}=-\omega_1^{\flat}\eta^{\sharp}\omega_2^{\flat}-\omega_2^{\flat}\eta^{\sharp}\omega_1^{\flat}\hspace{1cm}\text{for} \ \omega_1,\omega_2\in\Omega^{2}(A).
\]
This equality is the content of \cite[Lemma~4.1]{GMS}, which the authors use to describe the action of the $L_\infty[1]$-isomorphism $e^{-{\sf M}}$ on Maurer-Cartan elements. 
 Our Lemma~\ref{lem:M2} is more general since it applies to arbitrary  homogeneous elements of $\Omega^{\bullet}(A)$, and it is a key step in the proof of Prop.~\ref{prop:subalgebras} below.
\end{remark}

Before proving the lemma, recall that there is a bi-grading on $C^\infty(T^\ast[2]A[1])$, as we now explain. Take coordinates $\{x_i\}$ on $M$ and let $\{\xi_a\}$ be fiberwise linear coordinates on $A$. Denote by $\{p_i\}$ and $\{\theta_a\}$ the associated momenta; these are fiberwise linear coordinates on $T^{*}M$ and $A^{*}$, respectively. We get a coordinate chart $\{x_i,\xi_a,p_i,\theta_a\}$ for $T^\ast[2]A[1]$ where
\begin{equation}\label{grading}
\text{deg}(x_i)=0,\hspace{0.5cm}\text{deg}(\xi_a)=1,\hspace{0.5cm}\text{deg}(p_i)=2,\hspace{0.5cm}\text{deg}(\theta_a)=1.
\end{equation}
In fact, this grading on $C^{\infty}(T^\ast[2]A[1])$ is the sum of two gradings, stemming from the fact that $T^{*}A$ is a double vector bundle. {This in turn uses a canonical symplectomorphism $T^{*}A\rightarrow T^{*}A^{*}$ known as the Legendre transformation \cite[\S 3.4]{Roytenberg1999PhDthesis}; in the above coordinates it is simply given by 
\[
T^{*}A\rightarrow T^{*}A^{*}:(x_i,\xi_a,p_i,\theta_a)\mapsto (x_i,-\theta_a,p_i,\xi_a).
\]}
The double vector bundle in question is then the following:
\[
\begin{tikzcd}[row sep=large]
T^{*}A\arrow{r}{(p_i,\xi_a)}\arrow{d}[swap]{(p_i,\theta_a)} & A^{*}\arrow[d]\\
A\arrow[r]&M
\end{tikzcd},
\]
where $(p_i,\xi_a)$ and the fiber coordinates of $T^{*}A\rightarrow A^{*}$ and $(p_i,\theta_a)$ are those of $T^{*}A\rightarrow A$. When considering a vector bundle as a graded manifold, the fiber coordinates have degree one; hence the total degree on $T^{*}A$ is indeed given by \eqref{grading} above. {In what follows, we denote the bi-degree of an element in $C^{\infty}(T^\ast[2]A[1])$ by $(\bullet,\bullet)$, where the first entry is the degree in the coordinates $(p_i,\theta_a)$ and the second entry is the degree in $(p_i,\xi_a)$.}

\begin{proof}[Proof of Lemma~\ref{lem:M2}]
We first argue that $\sf M_2$ is the only nonzero Taylor coefficient.
\begin{itemize}
    \item To see that $\sf M_1$ vanishes, we should check that $P\circ {\sf m}=0$. Note that for all $f\in C^\infty(T^\ast[2]A[1])$, we have that $P(f)$ vanishes exactly when $f$ has no components of bi-degree $(0,\bullet)$. {Recall that the canonical Poisson bracket $\{-,-\}$ has bi-degree $(-1,-1)$ (see \cite[Rem.~3.3.3]{Roytenberg1999PhDthesis}).} 
 Since $\eta\in\Gamma(\wedge^{2}A)$, it has bi-degree $(2,0)$,
	hence $\sf m=\{\eta,-\}$ has bi-degree $(1,-1)$. Consequently, functions in the image of $\sf m$ have no components of bi-degree $(0,\bullet)$, and therefore $P\circ {\sf m}=0$. 
    
    \item The vanishing of ${\sf M}_k$ for $k\geq 3$ holds by bi-degree reasons. Forms $\omega_1\in\Omega^{k}(A), \omega_2\in\Omega^{l}(A)$ and $\omega_3\in\Omega^{p}(A)$ have bi-degrees $(0,k), (0,l)$ and $(0,p)$ respectively when considered as elements of $C^\infty(T^\ast[2]A[1])$. Since the derivation ${\sf m}$ has bi-degree $(1,-1)$ and the Poisson bracket $\{-,-\}$ has bi-degree $(-1,-1)$, it follows that 
	$
    \{{\sf m}\omega_1,\omega_2\}
    $
	has bi-degree $(0,k+l-2)$. Hence, $\{\{{\sf m}\omega_1,\omega_2\},\omega_3\}$ is zero because it has a negative first degree.
\end{itemize}
To prove the formula \eqref{eq:M2bis}, 
we first show by induction that
\begin{equation}\label{eq:contraction}
\{X,\alpha\}=\iota_{X}\alpha\hspace{1cm}\text{for}\ X\in\Gamma(A),\ \alpha\in\Omega^{\bullet}(A).
\end{equation}
It clearly holds for $\alpha\in\Omega^{1}(A)$, since then $\{X,\alpha\}=\ldab X,\alpha\rdab=\iota_X\alpha$, {by the first equality in \eqref{eq:Roytenberg-correspondence}.} Assuming that the formula holds for $k$-forms, let $\alpha\in\Omega^{k}(A)$ and $\beta\in\Omega^{1}(A)$. We then have
\begin{align*}
\{X,\alpha\wedge\beta\}&=\{X,\alpha\}\beta+(-1)^{k}\alpha\{X,\beta\}\\
&=(\iota_X\alpha)\beta+(-1)^{k}\alpha(\iota_X\beta)\\
&=\iota_X(\alpha\wedge\beta),
\end{align*}
{proving \eqref{eq:contraction}.}
Now, to show that the equality \eqref{eq:M2bis} holds, we can assume that $\eta$ is decomposable, i.e. 
$\eta=\eta_1\wedge\eta_2$ for $\eta_1,\eta_2\in\Gamma(A)$. Note that for bi-degree reasons, we have
\[
{\sf M}_2(\omega_1\odot\omega_2)=P\{\{\eta,\omega_1\},\omega_2\}=\{\{\eta,\omega_1\},\omega_2\}.
\]
Since
\begin{align*}
\{\eta_1\wedge\eta_2,\omega_1\}&=-\big(\{\omega_1,\eta_1\}\eta_2+(-1)^{|\omega_1|}\eta_1\{\omega_1,\eta_2\}\big)\\
&=(-1)^{|\omega_1|}\{\eta_1,\omega_1\}\eta_2+\eta_1\{\eta_2,\omega_1\}\\
&=(-1)^{|\omega_1|}(\iota_{\eta_1}\omega_1)\eta_2+\eta_1(\iota_{\eta_2}\omega_1),
\end{align*}
taking a further bracket with $\omega_2$ we obtain
\begin{align}\label{eq:M2}
\{\{\eta_1\wedge\eta_2,\omega_1\},\omega_2\}&=(-1)^{|\omega_1|}\{(\iota_{\eta_1}\omega_1)\eta_2,\omega_2\}+\{\eta_1(\iota_{\eta_2}\omega_1),\omega_2\}\nonumber\\
&=-(-1)^{|\omega_1|+|\omega_1||\omega_2|}\big(\{\omega_2,\iota_{\eta_1}\omega_1\}\eta_2+(-1)^{|\omega_2|(|\omega_1|-1)}\iota_{\eta_1}\omega_1\{\omega_2,\eta_2\}\big)\nonumber\\
&\hspace{0.5cm}-(-1)^{|\omega_1||\omega_2|}\big(\{\omega_2,\eta_1\}\iota_{\eta_2}\omega_1+(-1)^{|\omega_2|}\eta_1\{\omega_2,\iota_{\eta_2}\omega_1\}\big)\nonumber\\
&=-(-1)^{|\omega_1|+|\omega_2|}\iota_{\eta_1}\omega_1\{\omega_2,\eta_2\}-(-1)^{|\omega_1||\omega_2|}\{\omega_2,\eta_1\}\iota_{\eta_2}\omega_1,
\end{align}
using that $\{-,-\}$ vanishes when both arguments lie in $\Gamma(A^{*})$. By \eqref{eq:contraction}, the expression \eqref{eq:M2} becomes
\begin{align}\label{eq:withsign}
(-1)^{|\omega_1|}\iota_{\eta_1}\omega_1\iota_{\eta_2}\omega_2+(-1)^{|\omega_1||\omega_2|+|\omega_2|}\iota_{\eta_1}\omega_2\iota_{\eta_2}\omega_1
&=(-1)^{|\omega_1|}\iota_{\eta_1}\omega_1\iota_{\eta_2}\omega_2-(-1)^{|\omega_1|}\iota_{\eta_2}\omega_1\iota_{\eta_1}\omega_2\nonumber\\
&=(-1)^{|\omega_1|}\big(\iota_{\eta_1}\omega_1\iota_{\eta_2}\omega_2-\iota_{\eta_2}\omega_1\iota_{\eta_1}\omega_2\big).
\end{align}
On the other hand, we have by definition (see \eqref{eq:sharp:property1}) that
\[
(\omega_{1}^{\flat}\wedge\omega_2^{\flat})\eta_1\wedge\eta_2=\iota_{\eta_1}\omega_1\iota_{\eta_2}\omega_2-\iota_{\eta_2}\omega_1\iota_{\eta_1}\omega_2,
\]
so comparing with \eqref{eq:withsign}, we see that the equality \eqref{eq:M2bis} holds. This proves the lemma. 
\end{proof}

\subsection{The Regular Poisson Case}
\label{sec:splitting_change}
We now show that the $L_{\infty}[1]$-algebra $(\frakX_{\calF}^\bullet(M)[2],\{\frakl^G_k\})$ of a symplectic foliation $(\calF,\omega)$ is independent of the choice of distribution $G$ complementary to $T\calF$. We will proceed as follows. Given two complements $G_0$ and $G_1$, we first use   \S\ref{subsec:Diraccase}  to find an isomorphism between the $L_{\infty}[1]$-algebras $(\frakX^\bullet(M)[2],\{\frakl^{G_0}_k\})$ and $(\frakX^\bullet(M)[2],\{\frakl^{G_1}_k\})$. We then show that this map restricts to the $L_{\infty}[1]$-subalgebras $(\frakX_{\calF}^\bullet(M)[2],\{\frakl^{G_0}_k\})$ and $(\frakX_{\calF}^\bullet(M)[2],\{\frakl^{G_1}_k\})$.

\subsubsection{A canonical $L_{\infty}[1]$-isomorphism}

Given a symplectic foliation $(\calF,\omega)$ with corresponding regular Poisson structure $\Pi$ on $M$, let $G_0$ and $G_1$ be complements to the characteristic distribution $T\calF$. 

In the following, we describe an explicit isomorphism between the $L_{\infty}[1]$-algebras $(\frakX^\bullet(M)[2],\{\frakl^{G_0}_k\})$ and $(\frakX^\bullet(M)[2],\{\frakl^{G_1}_k\})$. The idea is as follows: deforming $\gr\Pi$ inside $\bbT M$ using the complements $G_i\oplus G_i^{0}$ for $i=0,1$ yields two $L_{\infty}[1]$-algebra structures on $\Omega^{\bullet}(\gr\Pi)[2]$ that are isomorphic by Thm. \ref{theor:GMS}. Since the $L_{\infty}[1]$-algebras $(\frakX^\bullet(M)[2],\{\frakl^{G_i}_k\})$ are obtained from the latter via a strict isomorphism (see Prop. \ref{prop:Koszul_algebra:isomorphism}), we will end up with an isomorphism between $(\frakX^\bullet(M)[2],\{\frakl^{G_0}_k\})$ and $(\frakX^\bullet(M)[2],\{\frakl^{G_1}_k\})$. 

Given the direct sum $\bbT M=\gr\Pi\oplus(G_0\oplus G_0^{0})$, we start by expressing $G_1\oplus G_1^{0}$ in this decomposition. Note that, since both $G_0$ and $G_1$ are transverse to $T\calF$, there exists $\epsilon\in\Gamma(\text{Hom}(G_0,T\calF))$ such that
\begin{equation}\label{eq:G1}
G_1=\gr(\epsilon)=\{\epsilon(Y)+Y:Y\in G_0\}.
\end{equation}
This immediately implies that
\begin{equation}\label{eq:G1Ann}
G_1^{0}=\gr(-\epsilon^{*})=\{\alpha-\epsilon^{*}\alpha:\alpha\in T^{*}\calF\}.
\end{equation}

\begin{lemma}\label{lem:eta}
We have that $G_1\oplus G_1^{0}=\gr(\eta)$, where $\eta\in\Gamma(\wedge^{2}\gr\Pi)$ is determined by
\begin{equation}\label{eq:etasharp}
\eta^{\sharp}(\psi(X+\beta))=-\omega^{\flat}(\epsilon(X))+\epsilon^{*}\big(\omega^{\flat}(\epsilon(X))-\beta\big)+\epsilon(X),
\end{equation}
for $X+\beta\in G_0\oplus G_0^{0}$.
Here $\psi$ denotes the canonical identification between $G_0\oplus G_0^{0}$ and $(\gr\Pi)^{*}$, i.e.
\begin{equation}\label{eq:psi}
\psi:G_0\oplus G_0^{0}\rightarrow(\gr\Pi)^{*}:X+\beta\mapsto\ldab X+\beta,-\rdab|_{\gr\Pi}.
\end{equation}
\end{lemma}

Before proving the lemma, we remark that the right hand side of \eqref{eq:etasharp} indeed belongs to $\Gamma(\gr\Pi)$, because $\Pi^{\sharp}$ vanishes on elements of $G_{0}^{*}\cong T\calF^{0}$ and $\Pi^{\sharp}\circ\omega^{\flat}=-\text{Id}_{T\calF}$.

\begin{proof}[Proof of Lemma \ref{lem:eta}]
Because of \eqref{eq:G1} and \eqref{eq:G1Ann}, we know that
\[
G_1\oplus G_1^{0}=\{\epsilon(Y)+Y+\alpha-\epsilon^{*}(\alpha): Y\in G_0,\  \alpha\in T^{*}\calF\}.
\]
Picking $Y\in G_0$ and $\alpha\in T^{*}\calF$, we write
\begin{equation}\label{eq:dec}
\epsilon(Y)+Y+\alpha-\epsilon^{*}(\alpha)=(X+\beta)+(\delta+\Pi^{\sharp}(\delta)),
\end{equation}
for some $X+\beta\in G_0\oplus G_0^{0}$ and $\delta+\Pi^{\sharp}(\delta)\in\gr\Pi$.
We have to show that expressing $\delta+\Pi^{\sharp}(\delta)$ in terms of $X$ and $\beta$ yields the right hand side of \eqref{eq:etasharp}.
The equality \eqref{eq:dec} implies that
\[
\epsilon(Y)+Y-X-\Pi^{\sharp}(\delta)=\beta+\delta-\alpha+\epsilon^{*}(\alpha)\in TM\cap T^{*}M,
\]
so both sides must be zero. On one hand, we then have
\[
\epsilon(Y)-\Pi^{\sharp}(\delta)=X-Y,
\]
where the left hand side belongs to $T\calF$ and the right hand side to $G_0$. Again, both must be zero, hence
\begin{equation}\label{eq:pisharpdelta}
\Pi^{\sharp}(\delta)=\epsilon(X)   
\end{equation}
and 
\begin{equation}\label{eq:projdelta}
\pr_{T^{*}\calF}\delta=-\omega^{\flat}(\Pi^{\sharp}(\delta))=-\omega^{\flat}(\epsilon(X)).
\end{equation}
On the other hand, we get that 
\[
\beta-\alpha+\pr_{T^{*}\calF}\delta=-\epsilon^{*}(\alpha)-\pr_{G_0^{*}}\delta,
\]
where the left hand side belongs to $T^{*}\calF\cong G_0^{0}$ and the right hand side to $G_0^{*}$. So both are zero, hence
\[
\alpha=\beta+\pr_{T^{*}\calF}\delta=\beta-\omega^{\flat}(\epsilon(X)),
\]
using \eqref{eq:projdelta} in the last equality. This in turn then implies that
\begin{equation}\label{eq:projdeltabis}
\pr_{G_0^{*}}\delta=-\epsilon^{*}(\alpha)=\epsilon^{*}\big(\omega^{\flat}(\epsilon(X))-\beta\big).
\end{equation}
Putting together \eqref{eq:pisharpdelta}, \eqref{eq:projdelta} and \eqref{eq:projdeltabis}, we obtain the conclusion of the lemma:
\begin{align*}
\delta+\Pi^{\sharp}(\delta)&=\pr_{T^{*}\calF}\delta+\pr_{G_0^{*}}\delta+\Pi^{\sharp}(\delta)\\
&=-\omega^{\flat}(\epsilon(X))+\epsilon^{*}\big(\omega^{\flat}(\epsilon(X))-\beta\big)+\epsilon(X).\qedhere
\end{align*}
\end{proof}

We are now in a position where we can apply the results of \S\ref{subsec:Diraccase} to the Dirac structure $\gr\Pi\subset \bbT M$ with almost Dirac complements $G_0\oplus G_0^{0}$ and $G_1\oplus G_1^{0}$. Using $\eta\in\Gamma(\wedge^{2}\gr\Pi)$ as defined in Lemma \ref{lem:eta}, we define the coderivation $\sf M$ of ${\sf S}(\Omega^\bullet(\gr\Pi)[2])$ with Taylor coefficients $\{{\sf M}_k\}$ given by
\begin{equation*}
{\sf M}_k(\omega_1\odot\ldots\odot\omega_k)=P\{\{\ldots\{{\sf m}\omega_1,\omega_2\},\ldots\},\omega_k\},
\end{equation*}
where $\{-,-\}$ is the Poisson bracket on $T^{*}[2]\gr\Pi[1]$, $\sf m=\{\eta,-\}$ and $P:C^{\infty}(T^{*}[2]\gr\Pi[1])\rightarrow\Omega^{\bullet}(\gr\Pi)$ is the restriction to the zero section. By Thm.~\ref{theor:GMS}, we can exponentiate $\sf M$ to an $L_{\infty}[1]$-isomorphism 
\[
e^{\sf M}:\left(\Omega^\bullet(\gr\Pi)[2],\mu_1^{G_0\oplus G_0^{0}},-\mu_2^{G_0\oplus G_0^{0}},\mu_3^{G_0\oplus G_0^{0}}\right)\rightarrow\left(\Omega^\bullet(\gr\Pi)[2],\mu_1^{G_1\oplus G_1^{0}},-\mu_2^{G_1\oplus G_1^{0}},\mu_3^{G_1\oplus G_1^{0}}\right).
\]
Along with Prop. \ref{prop:Koszul_algebra:isomorphism}, we obtain the following diagram of  $L_{\infty}[1]$-algebras and
$L_{\infty}[1]$-isomorphisms:
\begin{equation}\label{eq:diagram}
\begin{tikzcd}
 \left(\Omega^\bullet(\gr\Pi)[2],\mu_1^{G_0\oplus G_0^{0}},-\mu_2^{G_0\oplus G_0^{0}},\mu_3^{G_0\oplus G_0^{0}}\right)\arrow{r}{e^{\sf M}}\arrow{d}{\wedge^\bullet(\calR_\Pi|_{T^\ast M})^\ast}&\left(\Omega^\bullet(\gr\Pi)[2],\mu_1^{G_1\oplus G_1^{0}},-\mu_2^{G_1\oplus G_1^{0}},\mu_3^{G_1\oplus G_1^{0}}\right)\arrow{d}{\wedge^\bullet(\calR_\Pi|_{T^\ast M})^\ast}\\
 \left(\frakX^\bullet(M)[2],\frakl^{G_0}_1,-\frakl^{G_0}_2,\frakl^{G_0}_3\right)\arrow[r,dashed] & \left(\frakX^\bullet(M)[2],\frakl^{G_1}_1,-\frakl^{G_1}_2,\frakl^{G_1}_3\right)
\end{tikzcd},
\end{equation}
and the next step is to describe the induced isomorphism in the bottom row of this diagram. The isomorphism in question is obtained by first transporting $\eta$ along the identification $\calR_{-\Pi}:\gr\Pi\rightarrow T^{*}M$; this yields a two-form $\xi\in\Omega^{2}(M)$ which we now turn to describe.

\begin{lemma}\label{lem:ksi}
Define the two-form $\xi\in\Gamma(G_0^{*}\otimes T^{*}\calF\oplus\wedge^{2}G_0^{*})$ by the rules
\[
\begin{cases}
\xi(V,W)=0\hspace{2.68cm}\text{for}\ V,W\in\Gamma(T\calF),\\
\xi(V,W)=-\omega(V,\epsilon(W))\hspace{0.97cm}\text{for}\ V\in\Gamma(T\calF),\ W\in\Gamma(G_0),\\
\xi(V,W)=\omega(\epsilon(V),\epsilon(W))\hspace{0.75cm}\text{for}\ V,W\in\Gamma(G_0).
\end{cases}
\]
Then $\xi\in\Omega^{2}(M)$ corresponds with $\eta\in\Gamma(\wedge^{2}\gr\Pi)$ under the isomorphism $\calR_{-\Pi}:\gr\Pi\rightarrow T^{*}M$, i.e. 
\[
\wedge^{2}\calR_{-\Pi}(\eta)=\xi.
\]
\end{lemma}
\begin{proof}
We first prove two auxiliary results.

\vspace{0.1cm}

\underline{Claim 1:} $(\wedge^{2}\calR_{-\Pi}(\eta))^{\flat}=\calR_{-\Pi}\circ\eta^{\sharp}\circ(\calR_{-\Pi})^{*}$ as maps $TM\rightarrow T^{*}M$.

\vspace{0.1cm}
\noindent
To prove Claim 1, we may assume that $\eta$ is decomposable, i.e. $\eta=\eta_1\wedge\eta_2$ for some $\eta_1,\eta_2\in\Gamma(\gr\Pi)$. We then have for $V,W\in TM$:
\begin{align*}
\big\langle(\wedge^{2}\calR_{-\Pi}(\eta_1\wedge\eta_2))^{\flat}(V),W \big\rangle&=\big\langle\calR_{-\Pi}(\eta_1)\wedge\calR_{-\Pi}(\eta_2),V\wedge W \big\rangle\\
&=\begin{vmatrix}
\langle \calR_{-\Pi}(\eta_1),V\rangle & \langle\calR_{-\Pi}(\eta_1),W \rangle\\
\langle \calR_{-\Pi}(\eta_2),V\rangle & \langle \calR_{-\Pi}(\eta_2),W\rangle
\end{vmatrix}\\
&=\begin{vmatrix}
\langle \eta_1,(\calR_{-\Pi})^{*}(V)\rangle & \langle \eta_1,(\calR_{-\Pi})^{*}(W)\rangle\\
\langle \eta_2,(\calR_{-\Pi})^{*}(V)\rangle & \langle \eta_2,(\calR_{-\Pi})^{*}(W)\rangle
\end{vmatrix}\\
&=\Big\langle\big\langle\eta_1,(\calR_{-\Pi})^{*}(V) \big\rangle\eta_2-\big\langle\eta_2,(\calR_{-\Pi})^{*}(V)\big\rangle\eta_1,(\calR_{-\Pi})^{*}(W)\Big\rangle\\
&=\big\langle (\eta_1\wedge\eta_2)^{\sharp}((\calR_{-\Pi})^{*}(V)),(\calR_{-\Pi})^{*}(W) \big\rangle\\
&=\big\langle(\calR_{-\Pi}\circ(\eta_1\wedge\eta_2)^{\sharp}\circ(\calR_{-\Pi})^{*})(V),W \big\rangle.
\end{align*}

\vspace{0.1cm}

{The second auxiliary result uses again the map $\psi$ appearing in \eqref{eq:psi}.
Also recall from \S\ref{subsec:param} that $\gamma_{G_0}\in\Omega^{2}(M)$ is the extension of the leafwise symplectic form $\omega$ by zero on $G_0$.}

\vspace{0.1cm}

\underline{Claim 2:} 
The maps $\psi$ and $\calR_{-\Pi}$ are related as follows, for all $V\in\mathfrak{X}(M)$:
\[
(\calR_{-\Pi})^{*}(V)=\psi(\pr_{G_0}V+\gamma_{G_0}^{\flat}(V)).
\]

\vspace{0.1cm}
\noindent
To prove Claim 2, we note that for any $V\in\mathfrak{X}(M)$,  we can write
\begin{equation}\label{eq:assumption}
(\calR_{-\Pi})^{*}(V)=\psi(X+\beta)
\end{equation}
for some $X\in\Gamma(G_0)$ and $\beta\in\Gamma(G_0^{0})$, and we have to express $X$ and $\beta$ in terms of $V$. The equality \eqref{eq:assumption} implies that for all $\delta\in\Omega^{1}(M)$, we have
\[
\langle (\calR_{-\Pi})^{*}(V),\delta+\Pi^{\sharp}(\delta)\rangle = \langle \psi(X+\beta),\delta+\Pi^{\sharp}(\delta)\rangle,
\]
which amounts to
\[
\delta(V)=\delta(X-\Pi^{\sharp}(\beta)).
\]
Since this holds for all $\delta\in\Omega^{1}(M)$, we must have that $X-V-\Pi^{\sharp}(\beta)$ vanishes. Consequently,
\[
\begin{cases}
X=\pr_{G_0}V,\\
\beta=\gamma_{G_0}^{\flat}(-\Pi^{\sharp}(\beta))=\gamma_{G_0}^{\flat}(V-X)=\gamma_{G_0}^{\flat}(V).
\end{cases}
\]
 
\vspace{0.1cm}

We now prove the lemma. Using Claim 1 and Claim 2 above, we have for $V,W\in\mathfrak{X}(M)$:
\begin{align}\label{eq:wedge2}
(\wedge^{2}\calR_{-\Pi}(\eta))(V,W)&=\langle (\calR_{-\Pi}\circ\eta^{\sharp}\circ(\calR_{-\Pi})^{*})(V),W\rangle\nonumber\\
&=\big\langle (\calR_{-\Pi}\circ\eta^{\sharp}\circ\psi)(\pr_{G_0}V+\gamma_{G_0}^{\flat}(V)),W\big\rangle\nonumber\\
&=\left\langle \calR_{-\Pi}\Big(-\omega^{\flat}(\epsilon(\pr_{G_0}V))+\epsilon^{*}\big(\omega^{\flat}(\epsilon(\pr_{G_0}V))-\gamma_{G_0}^{\flat}(V)\big)+\epsilon(\pr_{G_0}V)\Big),W\right\rangle,
\end{align}
where the last equality holds by Lemma \ref{lem:eta}. Expanding the right hand side of \eqref{eq:wedge2} {using $\Pi^{\sharp}\circ\omega^{\flat}=-\text{Id}_{T\calF}$
}, we get
\[
(\wedge^{2}\calR_{-\Pi}(\eta))(V,W)=\left\langle -\omega^{\flat}(\epsilon(\pr_{G_0}V))+ \epsilon^{*}\big(\omega^{\flat}(\epsilon(\pr_{G_0}V))-\gamma_{G_0}^{\flat}(V)\big),W\right\rangle.
\]
In particular, we have:
\begin{itemize}
    \item If $V,W\in\Gamma(T\calF)$, then $(\wedge^{2}\calR_{-\Pi}(\eta))(V,W)=0$.
    \item If $V\in\Gamma(T\calF)$ and $W\in\Gamma(G_0)$, then
    \[
    (\wedge^{2}\calR_{-\Pi}(\eta))(V,W)=\big\langle -\epsilon^{*}(\omega^{\flat}(V)),W\big\rangle=-\omega(V,\epsilon(W)).
    \]
    \item If $V,W\in\Gamma(G_0)$, then
    \[
    (\wedge^{2}\calR_{-\Pi}(\eta))(V,W)=\big\langle -\omega^{\flat}(\epsilon(V))+\epsilon^{*}(\omega^{\flat}(\epsilon(V))),W\big\rangle=\omega(\epsilon(V),\epsilon(W)).
    \]
\end{itemize}
This shows that $(\wedge^{2}\calR_{-\Pi}(\eta))=\xi$, so the proof is finished.
\end{proof}

Using the two-form $\xi\in\Omega^{2}(M)$ defined in Lemma \ref{lem:ksi}, we define the coderivation ${\sf N}$ of ${\sf S}(\mathfrak{X}^\bullet(M)[2])$ with Taylor coefficients $\{{\sf N}_k\}$ given by
\begin{equation}\label{eq:Ncoeff}
{\sf N}_k(Q_1\odot\ldots\odot Q_k)=P\{\{\ldots\{{\sf n}Q_1,Q_2\},\ldots\},Q_k\},
\end{equation}
where $\{-,-\}$ is the Poisson bracket on $T^{*}[2]T^{*}M[1]$, ${\sf n}=\{\xi,-\}$ and $P:C^{\infty}(T^{*}[2]T^{*}M[1])\rightarrow\mathfrak{X}^{\bullet}(M)$ is the restriction to the zero section. As expected, the exponential $e^{{\sf N}}$ is exactly the dashed arrow in the diagram \eqref{eq:diagram} that we aim to describe.

\begin{lemma}\label{lem:expfits}
The coderivation ${\sf N}$ defined in \eqref{eq:Ncoeff} exponentiates to $e^{{\sf N}}:{\sf S}(\mathfrak{X}^\bullet(M)[2])\rightarrow{\sf S}(\mathfrak{X}^\bullet(M)[2])$. The latter fits in the following commutative diagram of $L_{\infty}[1]$-algebras and $L_{\infty}[1]$-isomorphisms: 
\[
\begin{tikzcd}
 \left(\Omega^\bullet(\gr\Pi)[2],\mu_1^{G_0\oplus G_0^{0}},-\mu_2^{G_0\oplus G_0^{0}},\mu_3^{G_0\oplus G_0^{0}}\right)\arrow{r}{e^{\sf M}}\arrow{d}{\wedge^\bullet(\calR_\Pi|_{T^\ast M})^\ast}&\left(\Omega^\bullet(\gr\Pi)[2],\mu_1^{G_1\oplus G_1^{0}},-\mu_2^{G_1\oplus G_1^{0}},\mu_3^{G_1\oplus G_1^{0}}\right)\arrow{d}{\wedge^\bullet(\calR_\Pi|_{T^\ast M})^\ast}\\
 \left(\frakX^\bullet(M)[2],\frakl^{G_0}_1,-\frakl^{G_0}_2,\frakl^{G_0}_3\right)\arrow{r}{e^{\sf N}} & \left(\frakX^\bullet(M)[2],\frakl^{G_1}_1,-\frakl^{G_1}_2,\frakl^{G_1}_3\right)
\end{tikzcd},
\]
\end{lemma}
The lemma can be obtained by functoriality from Lemma \ref{lem:Koszul_algebra:Courant_iso} and 
\S\ref{subsec:Diraccase}, but we provide a direct proof here. 
\begin{proof}
By Lemma $\ref{lem:M2}$, we know that ${\sf N}_{2}$ is the only nonzero Taylor coefficient of ${\sf N}$. By Prop.~\ref{prop:coalgebra_coderivation}, this implies that  ${\sf N}({\sf S}^n(\mathfrak{X}^\bullet(M)[2]))\subset{\sf S}^{n-1}(\mathfrak{X}^\bullet(M)[2])$. Hence, ${\sf N}$ is pronilpotent, so its exponential $e^{{\sf N}}$ is a well-defined graded coalgebra automorphism of ${\sf S}(\mathfrak{X}^\bullet(M)[2])$.

To show that the diagram commutes, it suffices to prove that
\[
\wedge^{\bullet}\calR_{\Pi}^{*}\circ{\sf M} \circ \wedge^{\bullet}\calR_{-\Pi}^{*}={\sf N},
\]
{which is an equality of self-maps on ${\sf S}(\mathfrak{X}^\bullet(M)[2])$}. This in turn will follow if we show that
\[
\wedge^{k+l-2}\calR_{\Pi}^{*}\left({\sf M}_{2}\left(\wedge^{k}\calR_{-\Pi}^{*}(Q_1)\odot\wedge^{l}\calR^{*}_{-\Pi}(Q_2)\right)\right)={\sf N}_{2}(Q_1\odot Q_2),\hspace{0.5cm}\forall  Q_1\in\mathfrak{X}^{k}(M),Q_2\in\mathfrak{X}^{l}(M).
\]
To see that this equality holds, note that 
\begin{align*}
\wedge^{k+l-2}\calR_{\Pi}^{*}\left({\sf M}_{2}\left(\wedge^{k}\calR_{-\Pi}^{*}(Q_1)\odot\wedge^{l}\calR^{*}_{-\Pi}(Q_2)\right)\right)&=(-1)^{k}\wedge^{k+l-2}\calR_{\Pi}^{*}\left(\big((\wedge^{k}\calR_{-\Pi}^{*}(Q_1))^{\flat}\wedge(\wedge^{l}\calR^{*}_{-\Pi}(Q_2))^{\flat}\big)\eta\right)\\
&=(-1)^{k}(Q_1^{\sharp}\wedge Q_2^{\sharp})(\wedge^{2}\calR_{-\Pi}(\eta))\\
&=(-1)^{k}(Q_1^{\sharp}\wedge Q_2^{\sharp})(\xi)\\
&={\sf N}_2(Q_1\odot Q_2),
\end{align*}
where the first and the last equality hold by Lemma \ref{lem:M2}, and the third equality uses Lemma \ref{lem:ksi}.
\end{proof}

\subsubsection{Restricting to $L_{\infty}[1]$-subalgebras}

In the previous lemma, we found an $L_{\infty}[1]$-isomorphism
\[
e^{{\sf N}}:\big(\frakX^\bullet(M)[2],\frakl^{G_0}_1,-\frakl^{G_0}_2,\frakl^{G_0}_3\big)\rightarrow\big(\frakX^\bullet(M)[2],\frakl^{G_1}_1,-\frakl^{G_1}_2,\frakl^{G_1}_3\big).
\]
It remains to show that it restricts to an isomorphism between the $L_{\infty}[1]$-subalgebras of good multivector fields. This is a consequence of the following lemma.

\begin{proposition}\label{prop:subalgebras}
The coderivation ${\sf N}$ of ${\sf S}(\mathfrak{X}^\bullet(M)[2])$ preserves ${\sf S}(\mathfrak{X}_{\good}^\bullet(M)[2])$. Consequently, we get an $L_{\infty}[1]$-isomorphism
\[
e^{{\sf N}}:\big(\frakX_{\good}^\bullet(M)[2],\frakl^{G_0}_1,-\frakl^{G_0}_2,\frakl^{G_0}_3\big)\rightarrow\big(\frakX_{\good}^\bullet(M)[2],\frakl^{G_1}_1,-\frakl^{G_1}_2,\frakl^{G_1}_3\big).
\]
\end{proposition}
\begin{proof}
We show that ${\sf N}$   preserves ${\sf S}(\mathfrak{X}_{\good}^\bullet(M)[2])$. Recall that, like any   coderivation of ${\sf S}(\mathfrak{X}^\bullet(M)[2])$,
  ${\sf N}$ is determined by its Taylor coefficients through the formula \eqref{eq:codercoeff}.
Since ${\sf N}_{2}$ is the only nonzero Taylor coefficient of ${\sf N}$ (see  Lemma \ref{lem:M2}), it suffices to show that for $Q_1\in\frakX_{\good}^{k}(M)$ and $Q_2\in\frakX_{\good}^{l}(M)$, we have
\[
{\sf N}_2(Q_1\odot Q_2)\in\frakX_{\good}^{k+l-2}(M).
\]
To this end, recall that the decomposition $TM=T\calF\oplus G_0$ endows the graded algebra $\mathfrak{X}^{\bullet}(M)$ with a bi-grading, defined as follows:	
\[
\mathfrak{X}^{\bullet}(M)=\bigoplus_{p,q\ge 0}\mathfrak{X}^{(p,q)}(M),\ \text{where}\ \mathfrak{X}^{(p,q)}(M):=\Gamma(\wedge^{p}T\calF\oplus\wedge^{q}G_0).
\]
In terms of this bi-grading, the space of good multivector fields is given by the following:
\[
\mathfrak{X}_{\good}^{\bullet}(M)=\bigoplus_{\substack{p\ge 0\\q=0,1}}\mathfrak{X}^{(p,q)}(M).
\]
The multivector fields $Q_1$ and $Q_2$ have {only components in} bi-degree $(k-i,i)$ and $(l-j,j)$ respectively, where $i,j\in\{0,1\}$. By Lemma \ref{lem:M2}, we know that
\begin{equation}\label{eq:bidegrees}
{\sf N}_2(Q_1\odot Q_2)=(-1)^{k}(Q_1^{\sharp}\wedge Q_2^{\sharp})\xi,
\end{equation}
and since the two-form $\xi$ belongs to $\Gamma(G_0^{*}\otimes T^{*}\calF\oplus\wedge^{2}G_{0}^{*})$ {by Lemma \ref{lem:ksi}}, the $(k+l-2)$-vector field \eqref{eq:bidegrees} has a component in bi-degree $(k-i+l-j-1,i+j-1)$ and a component in bi-degree $(k-i+l-j,i+j-2)$. Since $i+j-1$ and $i+j-2$ are at most equal to $1$, this shows that ${\sf N}_2(Q_1\odot Q_2)$ belongs to $\frakX_{\good}^{k+l-2}(M)$.
\end{proof}

Let $\tau$ be the transposition of $\frakX_{\good}^\bullet(M)[2]$ that flips the sign, which yields strict $L_{\infty}[1]$-isomorphisms 
\[
\tau:\big(\frakX_{\good}^\bullet(M)[2],\frakl^{G_i}_1,\frakl^{G_i}_2,\frakl^{G_i}_3\big)\rightarrow \big(\frakX_{\good}^\bullet(M)[2],\frakl^{G_i}_1,-\frakl^{G_i}_2,\frakl^{G_i}_3\big):Q\mapsto -Q.
\]
Pre- and postcomposing the map $e^{{\sf N}}$ from Prop.~\ref{prop:subalgebras} with $\tau$, we obtain the main result of this section.

\begin{corollary}\label{cor:preservessubalg}
The $L_{\infty}[1]$-algebra $(\frakX_{\calF}^\bullet(M)[2],\{\frakl^G_k\})$ associated with a symplectic foliation $(\calF,\omega)$ does not depend on the choice of complement $G$ to the characteristic distribution $T\calF$. Indeed, if $G_0$ and $G_1$ are two complements, then we have a canonical $L_{\infty}[1]$-isomorphism
\[
\tau\circ e^{{\sf N}}\circ\tau:\big(\frakX_{\good}^\bullet(M)[2],\frakl^{G_0}_1,\frakl^{G_0}_2,\frakl^{G_0}_3\big)\rightarrow\big(\frakX_{\good}^\bullet(M)[2],\frakl^{G_1}_1,\frakl^{G_1}_2,\frakl^{G_1}_3\big).
\]
\end{corollary}

	\section{Equivalences of Regular Poisson Structures}
	\label{sec:gauge_equivalence}
	
	This section addresses equivalences of deformations of a given regular Poisson structure $\Pi$ on $M$.  We show that the geometric notion of equivalence given by isotopies agrees with the algebraic notion of gauge equivalence obtained from the  {$L_{\infty}[1]$-algebra $(\frakX_\good^\bullet(M)[2],\{\frakl^G_k\})$ that governs the deformation problem of $\Pi$ (see Thm.~\ref{theor:deformation_theory regular_Poisson}). {We discuss the resulting moduli space of regular Poisson structures in an example, showing that it is not smooth in general.}

 \subsection{Gauge equivalence}
 {Our approach consists of relating the algebraic notion of gauge equivalence on $M$ with MC elements on the product manifold $M\times[0,1]^2$, and then -- on this product  -- of relating the MC elements of the $L_{\infty}[1]$-algebra 
introduced in Prop. \ref{prop:Koszul_algebra} with those of the Koszul dgL[1]a. This approach is analog to the one taken in \cite{SZequivalences}.}
	 We first recall a Poisson version of Moser's theorem, {see \cite[Thm. 9.48]{lectures} for a proof}.

	\begin{lemma}\label{moser}
	 Let $M$ be a compact manifold, $(\Pi_t)_{t\in[0,1]}$ a smooth path of Poisson structures on $M$ and $\phi_t$ an isotopy with time-dependent vector field $(Y_t)_{t\in[0,1]}$. Then $\Pi_t$ is generated by $\phi_t$,
	 i.e.
	 \[
	 \Pi_t=(\phi_t)_{*}\Pi_{0}
	 \]
	 exactly when
	 \[
	 \frac{d}{dt}\Pi_{t}=d_{\Pi_{t}}Y_{t}.
	 \]
	\end{lemma}

\noindent
If $\Pi_{t}$ is a smooth path of Poisson structures, then $\frac{d}{dt}\Pi_{t}$ defines a class in the second Poisson cohomology group $H^{2}_{\Pi_{t}}(M)$, since
		\begin{equation*}
			0=\tfrac{d}{dt}[\Pi_{t},\Pi_{t}]_{SN}=2\Big[\Pi_{t},\frac{d}{dt}\Pi_{t}\Big]_{SN}.
		\end{equation*}
Hence, Lemma~\ref{moser} requires that these classes vanish and that there exists a smooth path of primitives.

\bigskip
		
		In the following, we fix a regular Poisson structure $\Pi$ {on a manifold $M$}  with characteristic distribution $T\calF$, and we choose a splitting $TM=T\calF\oplus G$. 
		We first specialize the general notion of gauge equivalence of MC elements (see Def.~\ref{def:gauge}) to the    $L_{\infty}[1]$-algebra   $(\frakX^\bullet(M)[2],\{\frakl^G_k\})$ introduced in Prop.~\ref{prop:Koszul_algebra}.
		Recall that the MC equation (see Def.~\ref{def:mc}) for $W\in\mathfrak{X}^{2}(M)$ reads 
		\begin{equation}\label{eq:mceq}
			[\Pi,W]_{SN}+\frac{1}{2}[W,W]_{\gamma}{-}\frac{1}{6}(W^\sharp\wedge W^\sharp\wedge W^\sharp)\Upsilon^G_{TM}=0. 
		\end{equation}
			
		\begin{definition}
				Two MC elements  $W_{0},W_{1}$ of $\left(\mathfrak{X}^{\bullet}(M)[2],\frakl^G_{1},\frakl^G_{2},\frakl^G_{3}\right)$  are \textit{gauge equivalent} 
				if there is  a smooth family of MC elements $\left(W_{t}\right)_{t\in[0,1]}$ interpolating between them and  a smooth family $\left(X_{t}\right)_{t\in[0,1]}$ in $\mathfrak{X}(M)$ such that
				\begin{align}\label{mc}
					\frac{d}{dt}W_{t}&=\frakl^G_{1}(X_{t})+\frakl^G_{2}(X_{t},W_{t})+\frac{1}{2}\frakl^G_{3}(X_{t},W_{t},W_{t})\nonumber\\
					&=[\Pi,X_{t}]_{SN}-[X_{t},W_{t}]_{\gamma}{-}\frac{1}{2}\left(X_{t}^{\sharp}\wedge W_{t}^{\sharp}\wedge W_{t}^{\sharp}\right)\Upsilon^G_{TM}.
				\end{align}
		\end{definition}

		The next lemma allows us to phrase the gauge equivalence condition as a MC equation on a higher dimensional manifold. 
		Consider the manifold $M\times  I^2$, where $I=[0,1]$, endowed with the
		regular Poisson structure $\widetilde{\Pi}$ obtained as  the trivial lift of $\Pi$.
		Denote by $(t,s)$ the coordinates on $I^2$. The characteristic distribution $T\widetilde{\calF}$ of $\widetilde{\Pi}$ is just $T\calF$, and as a complement we can take   $\widetilde{G}=G\oplus \mathbb{R}\partial_s\oplus \mathbb{R}\partial_t$. Consequently, denoting by $p\colon M\times  I^2 \to M$ the projection, we have
		$\widetilde{\gamma}=p^*\gamma$ and  $\Upsilon^{\widetilde{G}}_{T(M\times I^2)}=p^*\Upsilon^G_{TM}$. 
		
		\begin{lemma}\label{lem:square}
			Let $\left(W_{t}\right)_{t\in[0,1]}$ be  smooth family in $\mathfrak{X}^2(M)$
			and  $\left(X_{t}\right)_{t\in[0,1]}$ a smooth family  in $\mathfrak{X}(M)$. 
			Then the $W_{t}$ are MC elements  of $(\frakX^\bullet(M)[2],\{\frakl^G_k\})$ 
			satisfying eq. \eqref{mc} if{f}
			\begin{equation}\label{eq:wtWt}
				\widetilde{W}_t:=W_t+(\partial_t+X_t)\wedge \partial_s
			\end{equation}
			is a MC element of $(\frakX^\bullet(M\times I^2)[2],\{\frakl^{\widetilde{G}}_{k}\})$, i.e.
			\begin{equation}\label{eq:mceqsquare}
				[\widetilde{\Pi},\widetilde{W}_t]_{SN}+\frac{1}{2}[\widetilde{W}_t,\widetilde{W}_t]_{\widetilde{\gamma}}
			{-}\frac{1}{6}(\widetilde{W}_t^\sharp\wedge \widetilde{W}_t^\sharp\wedge\widetilde{W}_t^\sharp)\Upsilon^{\widetilde{G}}_{T(M\times I^2)}=0. 
			\end{equation}
		\end{lemma}
		\begin{proof}
			Expanding eq. \eqref{eq:mceqsquare}, the terms not containing 
			$\partial_s$ give exactly the MC equation \eqref{eq:mceq} for $W_t$, and
			the terms  containing 
			$\partial_s$ give exactly the relation  \eqref{mc}. To see that this is the case, let us investigate each term in the left hand side of \eqref{eq:mceqsquare} separately.
			\begin{itemize}[leftmargin=0.2in]
			\item First, it is clear that
			    \[
			    [\widetilde{\Pi},\widetilde{W}_t]_{SN}=[\Pi,W_t]_{SN}+[\Pi,X_t]_{SN}\wedge\partial_s.
			    \]
		    \item Second, we have
			    \begin{align}\label{eq:secondterm}
			    \hspace{-0.05cm}[\widetilde{W}_t,\widetilde{W}_t]_{\widetilde{\gamma}}&=[W_t,W_t]_{\widetilde{\gamma}}+2[W_t,(\partial_t+X_t)\wedge\partial_s]_{\widetilde{\gamma}}+[(\partial_t+X_t)\wedge\partial_s,(\partial_t+X_t)\wedge\partial_s]_{\widetilde{\gamma}}\nonumber\\
			    &=[W_t,W_t]_{\widetilde{\gamma}}+2[W_t,\partial_t+X_t]_{\widetilde{\gamma}}\wedge\partial_s-2(\partial_t+X_t)\wedge[W_t,\partial_s]_{\widetilde{\gamma}}+2[\partial_t+X_t,\partial_s]_{\widetilde{\gamma}}\wedge(\partial_t+X_t)\wedge\partial_s.
			    \end{align}
			    We now make a few observations. For time-dependent vector fields $U_t,V_t\in\mathfrak{X}(M)$, we have
			    \begin{align*}
			 [U_t,V_t]_{\widetilde{\gamma}}&=[\pr_{\widetilde{G}}U_t,\pr_{\widetilde{G}}V_t]-\widetilde{\Pi}^{\sharp}\big(\calL_{\pr_{\widetilde{G}}U_t}\iota_{V_t}\widetilde{\gamma}-\calL_{\pr_{\widetilde{G}}V_t}\iota_{U_t}\widetilde{\gamma}\big)\\
			 &=[\pr_{G}U_t,\pr_{G}V_t]-\Pi^{\sharp}\big(\calL_{\pr_{G}U_t}\iota_{V_t}\gamma-\calL_{\pr_{G}V_t}\iota_{U_t}\gamma\big)\\
			 &=[U_t,V_t]_{\gamma},
			 \end{align*}
			 and by the Leibniz rule, this equality also holds for time-dependent bivector fields. Next, for a time-dependent vector field $V_t\in\mathfrak{X}(M)$, we have
			 \begin{align*}
			 [\partial_t,V_t]_{\widetilde{\gamma}}&=[\partial_t,\pr_{G}V_t]-\widetilde{\Pi}^{\sharp}\big(\calL_{\partial_t}\iota_{V_t}\widetilde{\gamma}-\calL_{\pr_{G}V_t}\iota_{\partial_t}\widetilde{\gamma}\big)\\
			 &=\pr_G\left(\frac{d}{dt}V_t\right)-\Pi^{\sharp}\left(\gamma^{\flat}\left(\frac{d}{dt}V_t\right)\right)\\
			 &=\pr_G\left(\frac{d}{dt}V_t\right)+\pr_{T\calF}\left(\frac{d}{dt}V_t\right)\\
			 &=\frac{d}{dt}V_t,
			 \end{align*}
			 and again by the Leibniz rule, this equality also holds for time-dependent bivector fields $V_t\in\mathfrak{X}^{2}(M)$. At last, it is clear that
			 \[
			 [\partial_s,V_t]_{\widetilde{\gamma}}=0
			 \]
			 for all time-dependent vector fields $V_t\in\mathfrak{X}(M)$. Using these observations, the equality \eqref{eq:secondterm} becomes
			 \[
			 [\widetilde{W}_t,\widetilde{W}_t]_{\widetilde{\gamma}}=[W_t,W_t]_{\gamma}-2\left(\frac{d}{dt}W_t+[X_t,W_t]_{\gamma}\right)\wedge\partial_s.
			 \]
			\item Third, we get
			\begin{align*}
			(\widetilde{W}_t^\sharp\wedge \widetilde{W}_t^\sharp\wedge\widetilde{W}_t^\sharp)\Upsilon^{\widetilde{G}}_{T(M\times I^2)}&=\big(W_t^{\sharp}+X_t^{\sharp}\wedge\partial_s\big)\wedge\big(W_t^{\sharp}+X_t^{\sharp}\wedge\partial_s\big)\wedge\big(W_t^{\sharp}+X_t^{\sharp}\wedge\partial_s\big)p^*\Upsilon^G_{TM}\\
			&=(W_t^{\sharp}\wedge W_t^{\sharp}\wedge W_t^{\sharp})\Upsilon^G_{TM}+3(X_t^{\sharp}\wedge W_t^{\sharp}\wedge W_t^{\sharp})\Upsilon^G_{TM}\wedge\partial_s.
			\end{align*}
			\end{itemize}
			Hence, collecting the terms that contain $\partial_s$ and those that don't, the equality \eqref{eq:mceqsquare}
			is equivalent with
			\[
			\begin{cases}
			[\Pi,W_t]_{SN}+\frac{1}{2}[W_t,W_t]_{\gamma}{-}\frac{1}{6}(W_t^{\sharp}\wedge W_t^{\sharp}\wedge W_t^{\sharp})\Upsilon^G_{TM}=0\\
			[\Pi,X_t]_{SN}-\left(\frac{d}{dt}W_t+[X_t,W_t]_{\gamma}\right){-}\frac{1}{2}(X_t^{\sharp}\wedge W_t^{\sharp}\wedge W_t^{\sharp})\Upsilon^G_{TM}=0
			\end{cases}.
			\]
			The first equality says that the $W_t$ are MC elements of $(\frakX^\bullet(M)[2],\{\frakl^G_k\})$ {by Equation \eqref{eq:mceq}}, while the second equality is exactly the evolution equation \eqref{mc} from the gauge equivalence. This finishes the proof.
		\end{proof}
		
		Let $(\mathfrak{X}^{\bullet}(M)[2],\frakm_1,\frakm_2)$ be the dgL[1]a corresponding with the Koszul dgLa $(\mathfrak{X}^{\bullet}(M)[1],d_{\Pi},[-,-]_{SN})$ under the décalage isomorphism \eqref{eq:decalage}. Explicitly, we have for $P,Q\in\mathfrak{X}^{\bullet}(M)$:
		\begin{align*}
		&\frakm_1(P)=d_{\Pi}(P),\\
		&\frakm_2(P,Q)=(-1)^{|P|}[P,Q]_{SN}.
		\end{align*}
		Lemma \ref{lem:square} also holds for the dgL[1]a's $(\mathfrak{X}^{\bullet}(M)[2],\frakm_1,\frakm_2)$
		and $(\frakX^\bullet(M\times I^2)[2],\widetilde{\frakm_1}, \widetilde{\frakm_2})$, as shown below.

		\begin{lemma}\label{lem:squarebis}
		 Let $\left(W_{t}\right)_{t\in[0,1]}$ be  smooth family in $\mathfrak{X}^2(M)$ and  $\left(X_{t}\right)_{t\in[0,1]}$ a smooth family  in $\mathfrak{X}(M)$. 
			Then
			\begin{equation}\label{eq:koszul}
				\begin{cases}
	  	[\Pi,W_{t}]_{SN}+\frac{1}{2}[W_t,W_t]_{SN}=0\\
					\frac{d}{dt}W_{t}=[\Pi+W_t,X_{t}]_{SN} 
				\end{cases}
			\end{equation}
			if{f} 
			$\widetilde{W}_t:=W_t+(\partial_t+X_t)\wedge \partial_s$   is a MC element of $(\frakX^\bullet(M\times I^2)[2],\widetilde{\frakm_1}, \widetilde{\frakm_2})$.
		\end{lemma}	
		\begin{proof}
		The bivector field $\widetilde{W}_t$ is a MC element  of $(\frakX^\bullet(M\times I^2)[2],\widetilde{\frakm_1}, \widetilde{\frakm_2})$ if{f} the following vanishes:
		\begin{align*}
		&\widetilde{\frakm_1}(\widetilde{W}_t)+\frac{1}{2}\widetilde{\frakm_2}(\widetilde{W}_t,\widetilde{W}_t)\\\
		&=[\Pi,W_t]_{SN}+[\Pi,X_t]_{SN}\wedge\partial_s+\frac{1}{2}\big([W_t,W_t]_{SN}+2[W_t,(\partial_t+X_t)\wedge\partial_s]_{SN}+[(\partial_t+X_t)\wedge\partial_s,(\partial_t+X_t)\wedge\partial_s]_{SN}\big)\\
		&=[\Pi,W_t]_{SN}+[\Pi,X_t]_{SN}\wedge\partial_s+\frac{1}{2}[W_t,W_t]_{SN}-\frac{d}{dt}W_t\wedge\partial_s+[W_t,X_t]_{SN}\wedge\partial_s\\
		&=[\Pi,W_t]_{SN}+\frac{1}{2}[W_t,W_t]_{SN}+\left([\Pi+W_t,X_t]_{SN}-\frac{d}{dt}W_t\right)\wedge\partial_s.\qedhere
		\end{align*}
		\end{proof}

		We claim that the assignment $Z\to Z^{\gamma}$	provides a bijection between the following sets:
		\begin{itemize}
			\item MC elements of $(\frakX^\bullet(M)[2],\{\frakl^G_k\})$ lying in $\calI_{\gamma}$,
			\item MC elements of $(\frakX^\bullet(M)[2],\frakm_1, \frakm_2)$ lying  in $\calI_{-\gamma}$.
		\end{itemize}
		Recall that $Z^{\gamma}$ is defined so that its sharp-map (see \eqref{eq:gammasharp})} is 	
		\begin{equation}\label{eq:Zgamma}
			(Z^{\gamma})^{\sharp}=Z^{\sharp}\circ(\id_{T^{*}M}+\gamma^{\flat}\circ Z^{\sharp})^{-1}.
		\end{equation}
		To see that the above claim holds, note that for $Z\in\calI_{\gamma}$ one has
		\[
		Z\in MC(\frakX^\bullet(M)[2],\{\frakl^G_k\})\ \Leftrightarrow\ \exp_G(Z)=\Pi+Z^{\gamma}\ \text{is Poisson},
		\]
		as a consequence of the correspondence \eqref{eq:assignMC}. 
		In turn, the bivector field $\Pi+Z^{\gamma}$ being Poisson is equivalent with $Z^{\gamma}$ being MC in $(\frakX^\bullet(M)[2],\frakm_1, \frakm_2)$.
		We now apply this bijection on the regular Poisson manifold $(M\times  I^2,\widetilde{\Pi})$, restricting ourselves to bivector fields of the same form as in eq. \eqref{eq:wtWt}.
		
		\begin{lemma}\label{lem:matrices}
			The bijection  $\widetilde{Z}\to (\widetilde{Z})^{\widetilde{\gamma}}$ between MC elements restricts to a bijection between
			\begin{itemize}
				\item MC elements of $(\frakX^\bullet(M\times I^2)[2],\{\frakl^{\widetilde{G}}_k\})$ of the form $$\widetilde{W}_t:=W_t+(\partial_t+X_t)\wedge \partial_s$$ for smooth families 
				$\left(W_{t}\right)_{t\in[0,1]}$ in $\calI_{\gamma}$ and  $\left(X_{t}\right)_{t\in[0,1]}$  in $\mathfrak{X}(M)$,
				\item MC elements of $(\frakX^\bullet(M\times I^2)[2],\widetilde{\frakm_1}, \widetilde{\frakm_2})$ of the form $$\widehat{W}_t+(\partial_t+\widehat{X}_t)\wedge \partial_s$$ for smooth families 
				$\big(\widehat{W}_{t}\big)_{t\in[0,1]}$ in $\calI_{-\gamma}$ and  
				$\big(\widehat{X}_{t}\big)_{t\in[0,1]}$  in $\mathfrak{X}(M)$.
			\end{itemize}
			Explicitly, we have 
			\begin{equation}\label{eq:correspond}
				\widehat{W}_t={W}_{t}^{\gamma},\quad\quad\quad \widehat{X}_t=(\id_{T M}+{W_t^\sharp}{{}\circ{}}{\gamma^\flat})^{-1}X_t.
			\end{equation}
		\end{lemma}
		\begin{proof}
			We have to compute explicitly $\widetilde{W}_t^{\widetilde{\gamma}}$, which is given by an expression of the type as in eq. \eqref{eq:Zgamma}. We first express
			$\widetilde{W}_t^{\sharp}\colon T^*(M\times I^2)\to T(M\times I^2)$
			as a block matrix, w.r.t. the decomposition of the domain as the direct sum of $T^*M$ and the lines spanned by $dt$ and $ds$, and of the codomain as the direct sum of $TM$ and the lines spanned by $\partial_t$ and $\partial_s$. We  obtain
			\begin{equation}\label{eq:W_t}
			\widetilde{W}_t^{\sharp}=\left(\begin{array}{c|c|c}W_t^{\sharp} & 0 & -X_t \\\hline 0 & 0 & -1 \\\hline X_t^{\sharp} & 1 & 0\end{array}\right).
			\end{equation}
			It follows that $\id_{T^* (M\times I^2)}+{\widetilde{\gamma}^\flat}{{}\circ{}}{\widetilde{W}_t^\sharp}$ is an upper triangular matrix:
			\begin{align*}
			\id_{T^* (M\times I^2)}+{\widetilde{\gamma}^\flat}{{}\circ{}}{\widetilde{W}_t^\sharp}&=\left(\begin{array}{c|c|c}\id_{T^{*}M} & 0 & 0 \\\hline 0 & 1 & 0 \\\hline 0 & 0 & 1\end{array}\right) + \left(\begin{array}{c|c|c}\gamma^{\flat}& 0 & 0 \\\hline 0 & 0 & 0 \\\hline 0 & 0 & 0\end{array}\right)\left(\begin{array}{c|c|c}W_t^{\sharp} & 0 & -X_t \\\hline 0 & 0 & -1 \\\hline X_t^{\sharp} & 1 & 0\end{array}\right)\\
			&=\left(\begin{array}{c|c|c}\id_{T^{*}M}+\gamma^{\flat}\circ W_t^{\sharp} & 0 & -\gamma^{\flat}(X_t) \\\hline 0 & 1 & 0 \\\hline 0 & 0 & 1\end{array}\right),
			\end{align*}
			which is invertible if{f} its diagonal entry
			$\id_{T^* M}+{\gamma^\flat}{{}\circ{}}{W_t^\sharp}$ is invertible. In this case, the inverse is  
			\begin{equation}\label{eq:inverse}
			\big(\id_{T^* (M\times I^2)}+{\widetilde{\gamma}^\flat}{{}\circ{}}{\widetilde{W}_t^\sharp}\big)^{-1}=\left(\begin{array}{c|c|c} 
				(\id_{T^* M}+{\gamma^\flat}{{}\circ{}}{W_t^\sharp})^{-1} & 0 &  (\id_{T^* M}+{\gamma^\flat}{{}\circ{}}{W_t^\sharp})^{-1}(\gamma^\flat(X_t)) \\
				\hline 0 & 1 & 0 \\
				\hline 0 & 0 & 1\end{array}\right).
			\end{equation}
			The matrix $\widetilde{W}_t^{\widetilde{\gamma}}$ is given by the product of the matrices \eqref{eq:W_t} and \eqref{eq:inverse}:
			\begin{equation}\label{eq:Wtwist}
			\big(\widetilde{W}_t^{\widetilde{\gamma}}\big)^{\sharp}=\left(\begin{array}{c|c|c} 
				W_t^{\sharp}\circ(\id_{T^* M}+{\gamma^\flat}{{}\circ{}}{W_t^\sharp})^{-1} & 0 &  W_t^{\sharp}\big((\id_{T^* M}+{\gamma^\flat}{{}\circ{}}{W_t^\sharp})^{-1}(\gamma^\flat(X_t))\big)-X_t \\
				\hline 0 & 0 & -1 \\
				\hline X_t^{\sharp}\circ(\id_{T^* M}+{\gamma^\flat}{{}\circ{}}{W_t^\sharp})^{-1} & 1 & X_t^{\sharp}\big((\id_{T^* M}+{\gamma^\flat}{{}\circ{}}{W_t^\sharp})^{-1}(\gamma^{\flat}(X_t))\big)\end{array}\right).
			\end{equation}
		We now check that this matrix is indeed of the form \eqref{eq:W_t}, involving $\widehat{W}_t$ and $\widehat{X}_t$ as specified in \eqref{eq:correspond}.
		\begin{itemize}
		    \item Clearly, the top left entry equals $(\widehat{W}_t)^{\sharp}$, by \eqref{eq:Zgamma}.
		    \item We check that the bottom right entry is zero. First, it is  clear that
		    \[
		    \gamma^{\flat}\circ(\id_{T M}+{W_t^\sharp}{{}\circ{}}{\gamma^\flat})=(\id_{T^* M}+{\gamma^\flat}{{}\circ{}}{W_t^\sharp})\circ\gamma^{\flat}.
		    \]
		    Since $\id_{T^* M}+{\gamma^\flat}{{}\circ{}}{W_t^\sharp}$ is invertible,  its dual $\id_{T M}+{W_t^\sharp}{{}\circ{}}{\gamma^\flat}$ is also invertible, and we get
		    \begin{equation}\label{eq:touse}
		    (\id_{T^* M}+{\gamma^\flat}{{}\circ{}}{W_t^\sharp})^{-1}\circ\gamma^{\flat}=\gamma^{\flat}\circ(\id_{T M}+{W_t^\sharp}{{}\circ{}}{\gamma^\flat})^{-1}.
		    \end{equation}
		    Using this, we obtain that
		    \begin{align*}
		   \left\langle(\id_{T^* M}+{\gamma^\flat}{{}\circ{}}{W_t^\sharp})^{-1}(\gamma^{\flat}(X_t)),X_t\right\rangle&=\left\langle \gamma^{\flat}\big((\id_{T M}+{W_t^\sharp}{{}\circ{}}{\gamma^\flat})^{-1}(X_t)\big),X_t\right\rangle\\
		   &=-\left\langle X_t,(\id_{T^* M}+{\gamma^\flat}{{}\circ{}}{W_t^\sharp})^{-1}(\gamma^{\flat}(X_t))\right\rangle
		    \end{align*}
		     {must vanish. This} confirms that the bottom right entry in \eqref{eq:Wtwist} is indeed zero.
		    \item The top right entry in \eqref{eq:Wtwist} reads
		    \[
		    \big(W_t^{\sharp}\circ(\id_{T^* M}+{\gamma^\flat}{{}\circ{}}{W_t^\sharp})^{-1}\circ\gamma^{\flat}-\id_{TM}\big)(X_t),
		    \]
		    so we have to show that
		    \begin{equation}\label{eq:toshowinverse}
		    (\id_{T M}+{W_t^\sharp}{{}\circ{}}{\gamma^\flat})^{-1}=\id_{T M}-W_t^{\sharp}\circ(\id_{T^* M}+{\gamma^\flat}{{}\circ{}}{W_t^\sharp})^{-1}\circ\gamma^{\flat}.
		    \end{equation}
		    To this end, we argue similarly as in the previous bullet point. Since
		    \[
		    (\id_{T M}+{W_t^\sharp}{{}\circ{}}{\gamma^\flat})\circ W_t^{\sharp}=W_t^{\sharp}\circ(\id_{T^* M}+{\gamma^\flat}{{}\circ{}}{W_t^\sharp}),
		    \]
		    we have
		    \[
		    W_t^{\sharp}\circ(\id_{T^* M}+{\gamma^\flat}{{}\circ{}}{W_t^\sharp})^{-1}=(\id_{T M}+{W_t^\sharp}{{}\circ{}}{\gamma^\flat})^{-1}\circ W_t^{\sharp}.
		    \]
		    Using this, we get 
		    		    \begin{align*}
		    &(\id_{T M}+{W_t^\sharp}{{}\circ{}}{\gamma^\flat})\circ\big(\id_{T M}-W_t^{\sharp}\circ(\id_{T^* M}+{\gamma^\flat}{{}\circ{}}{W_t^\sharp})^{-1}\circ\gamma^{\flat}\big)\\
		    &\hspace{0.5cm}=(\id_{T M}+{W_t^\sharp}{{}\circ{}}{\gamma^\flat})\circ\big(\id_{T M}-(\id_{T M}+{W_t^\sharp}{{}\circ{}}{\gamma^\flat})^{-1}\circ W_t^{\sharp}\circ\gamma^{\flat}\big)\\
		    &\hspace{0.5cm}=\id_{T M}.
		    \end{align*}
		    This confirms that the equality \eqref{eq:toshowinverse} holds.
		    \item Concerning the bottom left entry, it suffices to note that
		    \[
		    (\widehat{X}_t)^{\sharp}=\big((\id_{T M}+{W_t^\sharp}{{}\circ{}}{\gamma^\flat})^{-1}X_t\big)^{\sharp}=X_t^{\sharp}\circ(\id_{T^* M}+{\gamma^\flat}{{}\circ{}}{W_t^\sharp})^{-1}. 
		    \]
		\end{itemize}
		This proves that the matrix of $\widetilde{W}_t^{\widetilde{\gamma}}$ is of the form \eqref{eq:W_t}, with $\widehat{W}_t$ and $\widehat{X}_t$ as defined in \eqref{eq:correspond}.
		\end{proof}
		
		Combining Lemma \ref{lem:square},   Lemma \ref{lem:squarebis} and Lemma \ref{lem:matrices}, we obtain the following.
		\begin{proposition}\label{prop:combine}
			The equation~\eqref{eq:correspond} establishes a bijection between
			\begin{itemize}
				\item  Smooth families  $\left(W_{t}\right)_{t\in[0,1]}$ of  MC elements  of $(\frakX^\bullet(M)[2],\{\frakl^G_k\})$  lying in $\calI_{\gamma}$, 
				together with smooth families $\left(X_{t}\right)_{t\in[0,1]}$  in $\mathfrak{X}(M)$ satisfying eq. \eqref{mc}
				\item Smooth families  $\big(\widehat{W}_{t}\big)_{t\in[0,1]}$ 
				lying in $\calI_{-\gamma}$, 
				together with smooth families $\big(\widehat{X}_{t}\big)_{t\in[0,1]}$  in $\mathfrak{X}(M)$ satisfying eq. \eqref{eq:koszul}.
			\end{itemize}
				\end{proposition}
		Notice that in terms   of $\Pi_t =\Pi+\widehat{W_t}$, the equation \eqref{eq:koszul} says that 
		\begin{equation}\label{eq:exact}
		\begin{cases}
		[\Pi_t,\Pi_t]_{SN}=0\\
		\frac{d}{dt}\Pi_{t}=[\Pi_t,\widehat{X_{t}}]_{SN}
		\end{cases}.
		\end{equation}
		Restricting  Prop.~\ref{prop:combine} to elements $W_t$ lying in $\mathfrak{X}^2_\good(M)$, we obtain \emph{regular} Poisson structures $$\Pi_t:=\Pi+({W}_t)^{\gamma}$$ whose symplectic foliation is transverse to $G$, by Thm.~\ref{theor:deformation_theory regular_Poisson}. Recall that $\mathfrak{X}_{\good}(M)=\mathfrak{X}(M)$, so the same restriction has no effect on vector fields.
	Invoking the Moser lemma \ref{moser}, we obtain the main result of this section, showing that the gauge equivalence relation in $(\frakX^\bullet_\good(M)[2],\{\frakl^G_k\})$ coincides with the geometric notion of equivalence given by isotopies, in the following precise sense.
	
		\begin{theorem}\label{thm:gauge}
			Assume the  manifold $M$ is compact. Let $W_0$, $W_1$ be MC elements  of $(\frakX^\bullet_\good(M)[2],\{\frakl^G_k\})$  lying in $\calI_{\gamma}$. Then the following statements are equivalent:
			\begin{itemize}
				\item [i)] $W_0$ and $W_1$ are gauge equivalent through smooth families $\left(W_{t}\right)_{t\in[0,1]}\in\text{MC}(\frakX^\bullet_\good(M)[2],\{\frakl^G_k\})\cap\calI_{\gamma}$ and $\left(X_{t}\right)_{t\in[0,1]}\in\mathfrak{X}(M)$,
				\item [ii)]  there is a diffeomorphism $\psi$ which relates the corresponding Poisson structures
				$\Pi_0:=\Pi+(W_0)^{\gamma}$ and $\Pi_1:=\Pi+(W_1)^{\gamma}$, 
				and which is isotopic to $id_M$  by  an isotopy $(\psi_{t})_{t\in[0,1]}$ such that $(\psi_{t})_*(\im\Pi_0^{\sharp})$  is transverse to $G$ for all $t\in[0,1]$.
			\end{itemize}
		\end{theorem}
\begin{proof}
Given families $\left(W_{t}\right)_{t\in[0,1]}\in\text{MC}(\frakX^\bullet_\good(M)[2],\{\frakl^G_k\})\cap\calI_{\gamma}$ and $\left(X_{t}\right)_{t\in[0,1]}\in\mathfrak{X}(M)$ constituting a gauge equivalence between $W_0$ and $W_1$, Prop.~\ref{prop:combine} gives a path of regular Poisson structures $\Pi_t:=\Pi+({W}_t)^{\gamma}$ whose foliation is transverse to $G$, and a path of vector fields $\big(\widehat{X}_{t}\big)_{t\in[0,1]}\in\mathfrak{X}(M)$ such that
\[
\frac{d}{dt}\Pi_{t}=[\Pi_t,\widehat{X_{t}}]_{SN}.
\]
By Lemma~\ref{moser}, the path $\Pi_t$ is generated by the isotopy $(\psi_{t})_{t\in[0,1]}$ integrating $\big(\widehat{X}_{t}\big)_{t\in[0,1]}$, which then automatically satisfies $(\psi_{t})_*(\im\Pi_0^{\sharp})\pitchfork G$. 

Conversely, if $\Pi_0$ and $\Pi_1$ are related by an isotopy $(\psi_{t})_{t\in[0,1]}$ such that $(\psi_{t})_*(\im\Pi_0^{\sharp})\pitchfork G$, then Thm.~\ref{theor:deformation_theory regular_Poisson} implies that
\[
\Pi_t:=(\psi_t)_{*}\Pi_0=\Pi+(W_t)^{\gamma}
\]
for a smooth path $\left(W_{t}\right)_{t\in[0,1]}\in\text{MC}(\frakX^\bullet_\good(M)[2],\{\frakl^G_k\})\cap\calI_{\gamma}$ interpolating between $W_0$ and $W_1$. Moreover, Lemma~\ref{moser} shows that the time-dependent vector field $(Y_t)_{t\in[0,1]}\in\mathfrak{X}(M)$ of $(\psi_{t})_{t\in[0,1]}$ satisfies
\[
\frac{d}{dt}\Pi_t=[\Pi_t,Y_t].
\]
Applying Prop.~\ref{prop:combine} to $(\widehat{W}_t,\widehat{X}_{t})=((W_t)^{\gamma},Y_t)$, we see that the families $(W_t)_{t\in[0,1]}$ and $(X_t)_{t\in[0,1]}$ where
\[
X_t=(\id_{T M}+{W_t^\sharp}{{}\circ{}}{\gamma^\flat})(Y_t)
\]
constitute a gauge equivalence between $W_0$ and $W_1$. This finishes the proof.
\end{proof}

\subsection{The moduli space}
We now discuss the space $\RegPoiss^{2k}(M)$ of regular rank $2k$ Poisson structures on $M$, when quotienting by the action of the group of isotopies $\text{Diff}_{0}(M)$. By Thm.~\ref{theor:deformation_theory regular_Poisson} and Thm.~\ref{thm:gauge}, the formal tangent space to the moduli space $\RegPoiss^{2k}(M)/\text{Diff}_{0}(M)$ at the equivalence class of $\Pi$ is
\[
T_{[\Pi]}\big(\RegPoiss^{2k}(M)/\text{Diff}_{0}(M)\big)\cong H^{2}_{\good}(M,\Pi),
\]
where $H^{\bullet}_{\good}(M,\Pi)$ is the cohomology of the complex of good multivector fields $(\frakX^\bullet_\good(M),d_{\Pi})$, see Def.~\ref{def:good_multi-vector_fields}. The following example shows that this tangent space can change abruptly from point to point, which prevents the moduli space $\RegPoiss^{2k}(M)/\text{Diff}_{0}(M)$ from being smooth.

\begin{example}\label{ex:kronecker}
Consider the torus $\mathbb{T}^{3}$ with rank $2$ Poisson structure
\begin{equation*}
\Pi=\left(\partial_{\theta_1}+\lambda\partial_{\theta_2}\right)\wedge\partial_{\theta_3},
\end{equation*}
where $\lambda\in\mathbb{R}$ is a constant. The properties of the underlying foliation $T\mathcal{F}=\text{Span}\{\partial_{\theta_1}+\lambda\partial_{\theta_2},\partial_{\theta_3}\}$ depend on $\lambda$, for if $\lambda\in\mathbb{Q}$ then the leaves are embedded copies of $\mathbb{T}^{2}$, whereas for $\lambda\in\mathbb{R}\setminus\mathbb{Q}$ the leaves are immersed copies  of $\mathbb{R}\times S^{1}$ that are dense in $\mathbb{T}^{3}$.

Because $\Pi$ has corank one, the space of good multivector fields $\frakX^\bullet_\good(M)$ coincides with the space of all multivector fields $\frakX^\bullet(M)$, hence $H^{2}_{\good}(M,\Pi)$ is just the usual Poisson cohomology $H^{2}(M,\Pi)$. Because the Poisson structure $\Pi$ is induced by a cosymplectic structure, $H^{2}(M,\Pi)$ can be computed in terms of the foliated cohomology of $\mathcal{F}$, see \cite[Thm.~3.2.17]{Osorno}. We get
\[
H^{2}_{\good}(M,\Pi)=H^{2}(M,\Pi)\cong H^{2}(\mathcal{F})\oplus H^{1}(\mathcal{F}).
\]
Note that $(\mathbb{T}^{3},\mathcal{F})$ is a product foliation $(\mathbb{T}^{2},\mathcal{F}_{K})\times (S^{1},\mathcal{F}_{full})$, where $T\mathcal{F}_{K}=\text{Span}\{\partial_{\theta_1}+\lambda\partial_{\theta_2}\}$ and $\mathcal{F}_{full}$ is the one-leaf foliation on $S^{1}$. The second factor is compact and has finite dimensional foliated cohomology groups, which are just the de Rham cohomology groups of $S^1$. Hence, one can apply the K\"{u}nneth formula \cite{kunneth}
to compute $H^{\bullet}(\mathcal{F})$, which gives
\begin{align*}
&H^{2}(\mathcal{F})\cong H^{1}(\mathcal{F}_{K})\otimes H^{1}(\mathcal{F}_{full})\cong H^{1}(\mathcal{F}_{K}),\\
&H^{1}(\mathcal{F})\cong \big(H^{1}(\mathcal{F}_{K})\otimes H^{0}(\mathcal{F}_{full})\big) \oplus \big(H^{0}(\mathcal{F}_{K})\otimes H^{1}(\mathcal{F}_{full})\big)\cong H^{1}(\mathcal{F}_{K})\oplus H^{0}(\mathcal{F}_{K}).
\end{align*}
In conclusion, we have
\[
H^{2}_{\good}(M,\Pi)\cong H^{1}(\mathcal{F}_{K})\oplus H^{1}(\mathcal{F}_{K})\oplus H^{0}(\mathcal{F}_{K}).
\]
If $\lambda\in\mathbb{Q}$, then $\mathcal{F}_{K}$ is given by an $S^{1}$-fibration, hence the cohomology groups $H^{1}(\mathcal{F}_{K})$ and $H^{0}(\mathcal{F}_{K})$ are infinite dimensional. If $\lambda\in\mathbb{R}\setminus\mathbb{Q}$, then $\mathcal{F}_{K}$ is the well-known Kronecker foliation, and its first cohomology group $H^{1}(\mathcal{F}_{K})$ depends on which kind of irrational number $\lambda$ is \cite{haefliger},\cite[Chapter III]{liouville}. If $\lambda$ is a Liouville number\footnote{Recall that an irrational number $\lambda\in\mathbb{R}\setminus\mathbb{Q}$ is called Diophantine if there exist positive $s,c\in\mathbb{R}$ such that
\[
|m\lambda+n|\geq\frac{c}{(1+m^{2})^{s}}
\]
for any $(m,n)\in\mathbb{Z}^{2}\setminus\{(0,0)\}$. The Liouville numbers are the irrational numbers that are not Diophantine. They have the property that they are well-approximated by rational numbers.} then $H^{1}(\mathcal{F}_{K})$ is infinite dimensional, but if $\lambda$ is Diophantine then $H^{1}(\mathcal{F}_{K})$ is one-dimensional. Hence,
\begin{itemize}
\item if $\lambda\in\mathbb{Q}$ or $\lambda\in\mathbb{R}\setminus\mathbb{Q}$ is a Liouville number, then $H^{2}_{\good}(M,\Pi)$ is infinite dimensional.
\item if $\lambda\in\mathbb{R}\setminus\mathbb{Q}$ is Diophantine, then $H^{2}_{\good}(M,\Pi)\cong\mathbb{R}^{3}$ is finite dimensional.
\end{itemize}
This shows in particular that $H^{2}_{\good}(M,\Pi)$ is not stable under small perturbations of $\Pi$. For instance, let $\lambda$ be Diophantine and take a sequence of rational numbers $\lambda_n$ converging to $\lambda$. This gives rise to a family of Poisson structures $\Pi_n=\left(\partial_{\theta_1}+\lambda_{n}\partial_{\theta_2}\right)\wedge\partial_{\theta_3}$ converging to $\Pi$ in the $\mathcal{C}^{\infty}$-topology. The infinitesimal moduli space $T_{[\Pi]}\big(\RegPoiss^{2k}(M)/\text{Diff}_{0}(M)\big)$ is finite dimensional, whereas
$T_{[\Pi_n]}\big(\RegPoiss^{2k}(M)/\text{Diff}_{0}(M)\big)$ is infinite dimensional. Hence, $\RegPoiss^{2k}(M)/\text{Diff}_{0}(M)$ cannot be smooth.
\end{example}

		\appendix
		
		\section{Background on \texorpdfstring{$L_{\infty}[1]$-algebras}{L-algebras}}
		\label{app:L_infty_algebras}
		 		
	\subsection{$L_\infty[1]$-algebras}	
 We start by recalling the definition of an $L_{\infty}[1]$-algebra.

		\begin{definition}
		An \emph{$L_{\infty}[1]$-algebra} is a graded vector space $V$ with a collection of degree $1$ linear maps $\frakm_{k}:\otimes^{k}V\rightarrow V$ for $k\geq 1$, such that the following hold for homogeneous elements $v_1,\ldots,v_n\in V$.
		\begin{enumerate}
			\item Graded symmetry: for $\sigma\in S_n$, we have
			\[
			\frakm_{n}(v_{\sigma(1)},\ldots,v_{\sigma(n)})=\epsilon(\sigma;\bfv)\frakm_{n}(v_1,\ldots,v_n).
			\]
			\item Higher Jacobi identities: for all $n\geq 1$, we have
			\[
			\sum_{\substack{i+j=n+1\\ i,j\geq 1}}\sum_{\sigma\in S_{(i,n-i)}}\epsilon(\sigma;\bfv)\frakm_{j}\big(\frakm_{i}(v_{\sigma(1)},\ldots,v_{\sigma(i)}),v_{\sigma(i+1)},\ldots,v_{\sigma(n)}\big)=0.
			\] 
		\end{enumerate}

  Here $\epsilon(\sigma;\bfv)$ is the Koszul sign of $\sigma$ and $S_{(i,n-i)}$ denotes the space of $(i,n-i)$-unshuffles.
		\end{definition}
		
\begin{remark}
\begin{enumerate}
\item The notion of $L_{\infty}[1]$-algebra is equivalent to the more common notion of \emph{$L_{\infty}$-algebra} \cite{LadaMarkl}, in which the multibrackets $\frakl_{k}$ are graded skew-symmetric and of degree $2-k$. Given a graded vector space $V$, there is a correspondence between $L_{\infty}$-structures $\{\frakl_k\}$ on $V$ and $L_{\infty}[1]$-structures $\{\frakm_k\}$ on the shifted space $V[1]$ defined by $(V[1])^{i}=V^{i+1}$. The multibrackets are related as follows (we follow the convention of \cite[Rem.~1.1]{fiorenza2007cones}, up to a global minus sign):
\begin{equation}\label{eq:decalage}
\frakl_k(v_1,\ldots,v_k)=-(-)^{k}(-)^{\sum_{i}(k-i)|v_i|}\frakm_{k}(v_1[1],\ldots,v_k[1]).
\end{equation}
\item A dgL[1]a is an $L_{\infty}[1]$-algebra $(V,\{\frakm_k\})$ for which $\frakm_k=0$ whenever $k\geq 3$. The corresponding $L_{\infty}$-algebra then reduces to a differential graded Lie algebra (dgLa).
\end{enumerate}
\end{remark}		

A way to construct $L_{\infty}[1]$-algebras, which is important in this note, is Voronov's derived bracket construction. This approach uses the notion of V-data. 

\begin{definition}
A \emph{V-data} $(\frakh,\fraka,P,\Theta)$ consists of
\begin{itemize}
\item a graded Lie algebra $\frakh$ with Lie bracket $[-,-]$,
\item an abelian graded Lie subalgebra $\fraka\subset\frakh$,
\item a projection $P:\frakh\to\fraka$ such that $\ker P\subset\frakh$ is a graded Lie subalgebra,
\item a degree one element $\Theta\in\ker(P)$ such that $[\Theta,\Theta]=0$.
\end{itemize}
\end{definition}

\begin{proposition}[{\cite{voronov2005higher}}]
			\label{prop:higher_derived_brackets}
			A V-data $(\frakh,\fraka,P,\Theta)$ gives rise to an $L_\infty[1]$-algebra structure on $\fraka$, whose multibrackets $\calQ_k\in\Hom^1({\sf S}^k\fraka,\fraka)$ are given as follows, for all $a_1,\ldots,a_k\in\fraka$:
			\begin{equation*}
				\calQ_k(a_1\odot\cdots\odot a_k)=P[[\ldots[\Theta,a_1],\ldots],a_n].
			\end{equation*}
\end{proposition}

At last, we recall what are Maurer-Cartan elements (MC for short) and their gauge equivalences.
		
\begin{definition}\label{def:mc}
A \emph{Maurer-Cartan element} of an $L_{\infty}[1]$-algebra $(V,\{\frakm_k\})$ is an element $v$ of degree $0$ satisfying the Maurer-Cartan equation
\[
\sum_{k=1}^{\infty}\frac{1}{k!}\frakm_{k}(v,\ldots,v)=0.
\]
\end{definition}


\begin{definition}\label{def:gauge}
Two MC elements $v_0$ and $v_1$ of $(V,\{\frakm_k\})$ are called \emph{gauge equivalent} if there exist:
\begin{itemize}
    \item a smooth one-parameter family $(v_t)_{t\in[0,1]}$ of MC elements interpolating between $v_0$ and $v_1$,
    \item a smooth one-parameter family $(w_t)_{t\in [0,1]}$ of degree $-1$ elements of $V$,
\end{itemize}
such that the following evolution equation is satisfied:
\[
\frac{d}{dt}v_t=\sum_{k=0}^{\infty}\frac{1}{k!}\frakm_{k+1}(w_t,\underbrace{v_t,\ldots,v_t}_\text{$k$ times}).
\]
\end{definition}

The $L_{\infty}[1]$-algebras appearing in this note have only finitely many nonzero multibrackets, so their Maurer-Cartan equation and gauge equivalence are well-defined notions.

\subsection{$L_\infty[1]$-morphisms }
To define morphisms between $L_{\infty}[1]$-algebras, it is useful to view the latter as graded coalgebras equipped with a codifferential. We first recall this alternative point of view.		

Let $V=\oplus_{k\in\bbZ}V^k$ be a graded vector space.
Its graded symmetric algebra $\sf{S}^\bullet V:=\oplus_{k\in\bbN}\sf{S}^kV$ inherits from the tensor algebra ${\sf T}^\bullet V=\oplus_{k\in\bbN}V^{\otimes k}$ the structure of a graded coalgebra with coproduct $\mu$ given by
\begin{equation*}
	\mu(v_1\odot\ldots\odot v_n)=\sum_{i=1}^{n-1}\sum_{\sigma\in S_{(i,n-i)}}\epsilon(\sigma,{\bf v})(v_{\sigma(1)}\odot\ldots\odot v_{\sigma(i)})\otimes (v_{\sigma(i+1)}\odot\ldots\odot v_{\sigma(n)}).
\end{equation*}

\begin{definition}
A degree $k$ graded linear map $X:\sf{S}V\to \sf{S}V$ is a \emph{degree $k$ coderivation} of $(\sf{S}V,\mu)$ if
		\begin{equation*}
			\mu\circ X=(X\otimes\id+\id\otimes X)\circ\mu.
		\end{equation*}
\end{definition}

We now recall that coderivations can be described completely in terms of their Taylor coefficients.

		\begin{proposition}
			\label{prop:coalgebra_coderivation}
			For any graded vector space $V$, there is a degree $0$ graded linear isomorphism
			\begin{equation}\label{eq:isocoder}
				\operatorname{CoDer}^\bullet({\sf S} V)\longrightarrow \Hom^\bullet({\sf S}V,V)=\oplus_{n\in\bbN}\Hom^\bullet({\sf S}^n V,V)
			\end{equation}
			mapping each $\calQ\in\operatorname{CoDer}^k({\sf S}V)$ to the family $\{\calQ_n\}\in\oplus_{n\in\bbN}\Hom^k({\sf S}^nV,V)$ given by the following:
			\begin{equation*}
				\begin{tikzcd}
					{\sf S}^nV\arrow[rrr, bend right=15, swap, "\calQ_n"] \arrow[r, hook, "\text{incl}"]&{\sf S}^\bullet V\arrow[r, "\calQ"]&{\sf S}^\bullet V\arrow[r, two heads, "\pr_1"]&V
				\end{tikzcd},
			\end{equation*}
			where $\pr_1:{\sf S}^\bullet V\to V$ denotes the projection.
			In particular, $\calQ\in\operatorname{CoDer}^k({\sf S}V)$ can be reconstructed out of the family $\{\calQ_n\}\in\oplus_{n\in\bbN}\Hom^k({\sf S}^nV,V)$ as follows:
			\begin{equation}\label{eq:codercoeff}
				\calQ(v_1\odot\ldots\odot v_n)=\sum_{i=1}^n\sum_{\sigma\in S_{(i,n-i)}}\epsilon(\sigma;{\bf v})\calQ_i(v_{\sigma(1)}\odot\ldots\odot v_{\sigma(i)})\odot v_{\sigma(i+1)}\odot\ldots\odot v_{\sigma(n)},
			\end{equation}
			for all homogeneous $v_1,\ldots,v_n\in V$.
			Additionally, the isomorphism \eqref{eq:isocoder} induces a bijection between
			\begin{itemize}
				\item \emph{codifferentials} $\calQ$ of ${\sf S}V$, i.e.~$\calQ\in\operatorname{CoDer}^1({\sf S}V)$ such that $\calQ\circ\calQ=0$,
				\item $L_\infty[1]$-algebra structures $\{\calQ_n\}$ on $V$.
			\end{itemize}
		\end{proposition}
		
Given graded vector spaces $V$ and $W$, let $\mu_V$ and $\mu_W$ denote the coproducts on $\sf{S}V$ and $\sf{S}W$.
		A \emph{degree $k$ graded coalgebra morphism} $\sf{S}V\to \sf{S}W$ is a degree $k$ graded linear map $\Phi:\sf{S}V\to \sf{S}W$ such that
		\begin{equation*}
			(\Phi\otimes\Phi)\circ\mu_V=\mu_W\circ\Phi.
		\end{equation*}

With Prop. \ref{prop:coalgebra_coderivation} in mind, it is now clear how one should define morphisms between $L_{\infty}[1]$-algebras.
	
\begin{definition}
Let $(V,\{\calQ_n\})$ and $(W,\{\calR_n\})$ be $L_\infty[1]$-algebras, with codifferentials $\calQ$ on ${\sf S}V$ and $\calR$ on ${\sf S}W$.
An \emph{$L_\infty[1]$-algebra (iso)morphism} $(V,\{\calQ_n\})\longrightarrow(W,\{\calR_n\})$ is a degree $0$ graded coalgebra (iso)morphism $\Phi:{\sf S}V\to{\sf S}W$ such that $\calR\circ\Phi=\Phi\circ\calQ$.
\end{definition}	

The following proposition describes graded coalgebra morphisms in terms of their Taylor coefficients.

	\begin{proposition}
			\label{prop:coalgebra_morphism}
			For graded vector spaces $V$ and $W$, there is a degree $0$ graded linear isomorphism
			\begin{equation*}
				\Hom^\bullet({\sf S} V,{\sf S} W)\longrightarrow \Hom^\bullet({\sf S}V,W)=\oplus_{n\in\bbN}\Hom^\bullet({\sf S}^n V,W)
			\end{equation*}
			mapping each $\Phi\in\Hom^k({\sf S}V,{\sf S}W)$ to the family $\{\Phi_n\}\in\oplus_{n\in\bbN}\Hom^k({\sf S}^nV,W)$ given by the following:
			\begin{equation*}
				\begin{tikzcd}
					{\sf S}^nV\arrow[rrr, bend right=15, swap, "\Phi_n"] \arrow[r, hook, "\text{incl}"]&{\sf S}^\bullet V\arrow[r, "\Phi"]&{\sf S}^\bullet W\arrow[r, two heads, "\pr_1"]&W
				\end{tikzcd}.
			\end{equation*}
			In particular, $\Phi$ can be reconstructed out of the family $\{\Phi_n\}$ as follows:
			\begin{align*}
				&\Phi(v_1\odot\ldots\odot v_n)\\
				&\hspace{1cm}=\sum_{i=1}^n\sum_{p_1+\ldots+p_i=n}\sum_{\sigma\in{\sf S}_n}\frac{\epsilon(\sigma;{\bf v})}{i!p_1!\cdots p_i!}\Phi_{p_1}(v_{\sigma(1)}\odot\ldots\odot v_{\sigma(p_1)})\odot\cdots\odot \Phi_{p_i}(v_{\sigma(p_{1}+\cdots+p_{i-1}+1)}\odot\ldots\odot v_{\sigma(n)})
			\end{align*}
			for all homogeneous $v_1,\ldots,v_n\in V$.
			Moreover, $\Phi:{\sf S}V\to{\sf S}W$ is invertible exactly when $\Phi_1:V\to W$ is.
		\end{proposition}

\begin{remark}
An $L_{\infty}[1]$-algebra morphism $(V,\{\calQ_n\})\longrightarrow(W,\{\calR_n\})$ is \emph{strict} if the Taylor coefficients $\Phi_n$ of the corresponding codifferential coalgebra morphism $\Phi:({\sf S}V,\calQ)\to({\sf S}W,\calR)$  vanish for $n>1$.
\end{remark}

\section{Courant Algebroids and Dirac Structures}\label{app:Diracgeom}

We review some background material concerning Courant algebroids and Dirac structures. 

\begin{definition}
	A \emph{Courant algebroid} is a vector bundle $E\to M$  equipped with a non-degenerate symmetric pairing $\ldab-,-\rdab:E\otimes E\to\bbR$, a bilinear bracket $\ldsb-,-\rdsb:\Gamma(E)\times\Gamma(E)\to\Gamma(E)$ called the \emph{Dorfman bracket}, and a vector bundle morphism $\rho:E\to TM$ over $\id_M$ called the \emph{anchor map}, safisfying
	\begin{align*}
		&\ldsb e,\ldsb e^\prime,e^{\prime\prime}\rdsb\rdsb=\ldsb\ldsb e,e^\prime\rdsb,e^{\prime\prime}\rdsb+\ldsb e^\prime,\ldsb e,e^{\prime\prime}\rdsb\rdsb,\\
		&\ldab\ldsb e,e^\prime\rdsb,e^{\prime\prime}\rdab=\ldab e,\ldsb e^\prime,e^{\prime\prime}\rdsb\rdab,\\
		&\rho(e)\ldab e^\prime,e^\prime\rdab=2\ldab\ldsb e,e^\prime\rdsb,e^\prime\rdab,
	\end{align*}
	for any $e,e^\prime,e^{\prime\prime}\in\Gamma(E)$.
	Consequently, the Dorfman bracket and the anchor are related by the following Leibniz rule, for $e,e^\prime\in\Gamma(E)$ and $f\in C^\infty(M)$:
	\begin{equation*}
		\ldsb e,fe^\prime\rdsb=(\rho(e)f)e^\prime+f\ldsb e,e^\prime\rdsb.
	\end{equation*}
\end{definition}

\begin{example}
	The \emph{generalized tangent bundle} $\bbT M:=TM\oplus T^\ast M$ is the prototypical example of a Courant algebroid. Its pairing $\ldab-,-\rdab$, Dorfman bracket $\ldsb-,-\rdsb$ and anchor $\rho$ are defined on sections $X+\alpha,Y+\beta\in\Gamma(\bbT M)$ as follows:
	\begin{equation*}
		\ldab X+\alpha,Y+\beta\rdab=\alpha(Y)+\beta(X),\quad\rho(X+\alpha)=X,\quad\ldsb X+\alpha,Y+\beta\rdsb=[X,Y]+\calL_X\beta-\iota_Y\rmd\alpha.
	\end{equation*}
\end{example}

\begin{definition}
	Given a Courant algebroid $E$, a subbundle $L\subset E$ is called an \emph{almost Dirac structure} if $L=L^\perp$, where $L^\perp\subset E$ denotes the orthogonal of $L$ w.r.t. the pairing $\ldab-,-\rdab$.
	A \emph{Dirac structure} is an almost Dirac structure $L\subset E$ that is additionally involutive w.r.t. the Dorfman bracket $\ldsb-,-\rdsb$. 
\end{definition}

\begin{remark}
	\label{rem:Courant_tensor}
	For each almost Dirac structure $L\subset E$, its \emph{Courant tensor} $\Upsilon_L\in\Gamma(\wedge^3 L^{*})$ is defined by
	\begin{equation*}
		\Upsilon_L(\xi_1,\xi_2,\xi_3)=\ldab\xi_1,\ldsb\xi_2,\xi_3\rdsb\rdab,
	\end{equation*}
	for all $\xi_1,\xi_2,\xi_3\in\Gamma(L)$. It is easy to see that $L$ is Dirac if and only if $\Upsilon_L=0$.
\end{remark}

\begin{example}
Poisson structures on $M$ can be viewed as Dirac structures in $\bbT M$, as follows. Bivector fields  on $M$ correspond with almost Dirac structures that are transverse to $TM$, via
\begin{equation*}
\begin{aligned}
\frakX^2(M)&\overset{\sim}{\longrightarrow}\{L\subset\bbT M\ \text{almost Dirac structure}\mid L\pitchfork TM\},\\
Z&\longmapsto\gr(Z):=\{\iota_\alpha Z+\alpha\mid\alpha\in T^\ast M\}.
\end{aligned}
\end{equation*}   
Under this correspondence, $Z\in\frakX^2(M)$ is Poisson if and only if $\gr(Z)\subset\bbT M$ is Dirac.
\end{example}

\begin{remark}
	\label{rem:almost_Dirac_structure}
	Let $L\subset E$ be an almost Dirac structure, and assume that $R$ is a complementary subbundle of $E$. Using the projection $\pr_{L}$ with kernel $R$, one can restrict the Dorfman bracket $\ldsb-,-\rdsb$ to an almost Lie bracket $\pr_{L}\circ\ldsb-,-\rdsb$ on $\Gamma(L)$. Along with the restriction of the anchor map $\rho$ to $L$, we get an almost Lie algebroid structure $(\rho|_{L},\pr_{L}\circ\ldsb-,-\rdsb)$ on $L$.
	If $L$ is in fact Dirac, then it inherits an honest Lie algebroid structure, independent of $R$.
\end{remark}

		\section{Deformation Theory of Dirac Structures}
		\label{app:Deformations_Dirac_Structures}

		We review some constructions and results from deformation theory of Dirac structures. Given a Courant algebroid $E\to M$ and a Dirac structure $A\subset E$, we recall how the derived bracket construction in Prop.~\ref{prop:higher_derived_brackets} can be used to obtain an $L_{\infty}[1]$-algebra structure
		on $\Omega^{\bullet}(A)[2]$.

		Start by fixing an almost Dirac structure $B\subset E$ complementary to $A$, so that $E=A\oplus B$.
		Then the pairing $\ldab-,-\rdab$ induces a vector bundle isomorphism
		\begin{equation}\label{eq:VB-iso_A+A^ast}
			\phi_B:E=A\oplus B\overset{\sim}{\longrightarrow} A\oplus A^\ast,\ u+v\longmapsto u+\ldab v,-\rdab|_A,
		\end{equation}
		which allows us to transfer the Courant algebroid structure from $E=A\oplus B$ to $A\oplus A^\ast$. In particular, the non-degenerate symmetric pairing on $A\oplus A^\ast$ is given by
		\begin{equation*}
			\ldab X+\alpha,Y+\beta\rdab=\alpha(Y)+\beta(X).
		\end{equation*}
		By \cite[Thm.~4.5]{Roytenberg2002gradedsymplectic}, the Courant algebroid structure thus obtained on $A\oplus A^{*}$ corresponds to a degree $2$ symplectic NQ-manifold $(T^\ast[2]A[1],\omega,\Delta)$, where $\omega$ is the canonical degree $2$ symplectic structure on the graded manifold $T^\ast[2]A[1]$ and $\Delta=\{\Theta_B,-\}$
		for some MC element $\Theta_B$ of the graded Lie algebra $(C^\infty(T^\ast[2]A[1])[2],\{-,-\})$. Here $\{-,-\}$ is the degree $-2$ Poisson bracket corresponding with $\omega$.
		Explicitly, one has
		\begin{equation}\label{eq:Roytenberg-correspondence}
			\ldab u, v\rdab=\{u,v\},\quad\ldsb u, v\rdsb=\{\{u,\Theta_B\},v\},\quad\rho(u)f=\{\{u,\Theta_B\},f\},
		\end{equation}
		for all $u,v\in\Gamma(A\oplus A^\ast)\subset C^\infty(T^\ast[2]A[1])$ and $f\in C^\infty(M)\subset C^\infty(T^\ast[2]A[1])$. 
		
		According to \cite[Lemma 2.6]{FZgeo}, we can now cook up the V-data
		$
			(\frakh,\fraka,P,\Theta_B)
		$
		consisting of:
		\begin{enumerate}
			\item the graded Lie algebra $\frakh:=(C^\infty(T^\ast[2]A[1])[2],\{-,-\})$,
			\item the abelian Lie subalgebra $\fraka:=C^\infty(A[1])[2]=\Omega^\bullet(A)[2]$,
			\item the projection $P:\frakh\to\fraka$ is the restriction to the zero section, satisfying $\{\ker P,\ker P\}\subset\ker P$,
			\item the MC element $\Theta_B\in\operatorname{MC}(\frakh)$ satisfying $P(\Theta_B)=0$.
		\end{enumerate}	
		Applying Voronov's derived bracket construction (see Prop.~\ref{prop:higher_derived_brackets}), one obtains an $L_\infty[1]$-algebra structure $\{\mu_k^B\}$ on $\fraka=\Omega^\bullet(A)[2]$ whose multibrackets $\mu_k^B:{\sf S}^k\fraka\to\fraka$ are defined as follows:
		\begin{equation}
			\label{eq:higher_derived_brackets}
			\mu_k^B(\alpha_1,\ldots,\alpha_k):=P\{\{\ldots\{\Theta_B,\alpha_1\},\ldots,\},\alpha_k\}.
		\end{equation}
		Actually, as pointed out in~\cite[Lemma 2.6]{FZgeo}, $\mu_k^B=0$ for all $k>3$, and the other multibrackets $\mu^B_1,\mu^B_2,\mu^B_3$ can be re-expressed in terms of the Courant algebroid $E=A\oplus B$, as described in Lemma~\ref{lem:2.6first} below.
		
		Before stating Lemma~\ref{lem:2.6first}, we need to introduce some notation.
		Let $A\to M$ be a vector bundle.
		For any $k$-form $\alpha$ on $A$, i.e.~$\alpha\in\Omega^k(A)$, we denote by $\alpha^\flat:A\to\wedge^{k-1}A^\ast$ the contraction map.
		For $\alpha_1\in\Omega^{k_1}(A),\ldots,\alpha_n\in\Omega^{k_n}(A)$, we then define a bundle morphism $\alpha_1^\flat\wedge\ldots\wedge \alpha_n^\flat:\wedge^n A\to\wedge^{k_1+\ldots+k_n-n}A^\ast$ by
		\begin{equation}
			\label{eq:sharp:property1}
			(\alpha_1^\flat\wedge\ldots\wedge \alpha_n^\flat)(\xi_1\wedge\ldots\wedge\xi_n)=\sum_{\sigma\in S_n}(-)^\sigma \alpha_1^\flat(\xi_{\sigma(1)})\wedge\ldots\wedge \alpha_n^\flat(\xi_{\sigma(n)}).
		\end{equation}
		Using this morphism, 
		the next lemma re-expresses the multibrackets~\eqref{eq:higher_derived_brackets} in terms of the geometry of the Courant algebroid $E=A\oplus B$. We add an extra minus sign to the binary bracket, for reasons that we explain in Rem.~\ref{rem:decalage} below. {We remark that in \cite[Lemma 2.6]{FZgeo} a global minus sign  in front of the trinary bracket was mistakenly omitted.}
		\begin{lemma}{\cite[Lemma 2.6]{FZgeo}}
			\label{lem:2.6first}
			Given a Courant algebroid $E\to M$, let $E=A\oplus B$ be a splitting into a Dirac structure $A$ and an almost Dirac complement $B$.
						The graded vector space $\Omega^\bullet(A)[2]$ has an induced $L_\infty[1]$-algebra structure $\{\mu_k^B\}$ whose only non-trivial multibrackets $\mu_1^B,\ \mu_2^B,\ \mu_3^B$ are given as follows:
			\begin{itemize}
				\item $\mu_1^B$ coincides with the de Rham differential of Lie algebroid $A\Rightarrow M$, i.e.~for all $\alpha\in\Omega^\bullet(A)$:
				\begin{equation*}
					\mu_1^B(\alpha)=\rmd_A\alpha.
				\end{equation*}
				\item $\mu_2^B$ acts as follows, for all homogeneous $\alpha,\beta\in\Omega^\bullet(A)$:
				\begin{equation}
					\label{eq:lem:L_infty-algebra_Dirac:binary_bracket}
					\mu_2^B(\alpha,\beta)=(-1)^{|\alpha|}[\alpha,\beta]_B.
				\end{equation}
				Here $[-,-]_B$ is the almost Gerstenhaber bracket on $\Gamma(\wedge^{\bullet}B)$ defined by extending the almost Lie bracket $\pr_{B}\circ\ldsb-,-\rdsb$ of the almost Lie algebroid $B$.
				\item $\mu_3^B$ acts as follows, for all homogeneous $\alpha,\beta,\gamma\in\Omega^\bullet(A)$:
				\begin{equation}
					\label{eq:lem:L_infty-algebra_Dirac:ternary_bracket}
					\mu_3^B(\alpha, \beta, \gamma)={-}(-1)^{|\beta|} (\alpha^\flat\wedge \beta^\flat \wedge \gamma^\flat) \Upsilon_B.
				\end{equation}
				Here $\Upsilon_B$ denotes the Courant tensor of the almost Dirac structure $B$ (see Rem.~\ref{rem:Courant_tensor}).
			\end{itemize}
			In the RHS of equations~\eqref{eq:lem:L_infty-algebra_Dirac:binary_bracket} and~\eqref{eq:lem:L_infty-algebra_Dirac:ternary_bracket}, we use the identification $B\overset{\simeq}{\longrightarrow} A^\ast:\ u\longmapsto\ldab u,-\rdab|_A$.
		\end{lemma}
\noindent		
For exact Courant algebroids $E$, the above $L_\infty[1]$-algebra structure 
also appeared in  \cite[Cor.~3.7]{GMS}.

		\begin{remark}
			\label{rem:decalage}
			We list some remarks concerning the $L_{\infty}[1]$-algebra $\big(\Omega^\bullet(A)[2],\mu_1^{B},\mu_2^{B},\mu_3^{B}\big)$.
			\begin{enumerate}
				\item For any almost Dirac structure $B\subset E$ complementary to the Dirac structure $A\subset E$, the $L_\infty[1]$-algebra structure $\{\mu_k^B\}$ reduces to a dgL[1]a if and only if $B$ is Dirac. 
				\item Strictly speaking, the $L_{\infty}[1]$-algebra that one gets applying Voronov's derived bracket construction is $\big(\Omega^\bullet(A)[2],\mu_1^{B},-\mu_2^{B},\mu_3^{B}\big)$ , i.e. the one introduced in Lemma \ref{lem:2.6first} up to a global sign in the binary bracket. Note that changing the sign of the binary bracket produces an isomorphic $L_{\infty}[1]$-algebra; a strict $L_{\infty}[1]$-isomorphism is simply given by
				\[
				\big(\Omega^\bullet(A)[2],\mu_1^{B},\mu_2^{B},\mu_3^{B}\big)\rightarrow\big(\Omega^\bullet(A)[2],\mu_1^{B},-\mu_2^{B},\mu_3^{B}\big):\alpha\mapsto-\alpha.
				\]
				Introducing this additional minus sign has two advantages. First, it ensures that no minus sign appears in Lemma \ref{lem:2.6second} below. Second, in case the complement $B$ is Dirac, the $L_\infty$-algebra structure on $\Omega^\bullet(A)[1]$ corresponding with $\big(\Omega^\bullet(A)[2],\mu_1^{B},\mu_2^{B},\mu_3^{B}\big)$ under the décalage isomorphism \eqref{eq:decalage} reduces to the dgLa $(\Omega^\bullet(A)[1],d_A,[-,-]_{A^{*}})$ used in~\cite{liu1995manin}.
			\end{enumerate}
		\end{remark}

		Given a splitting $E=A\oplus B$ as in Lemma~\ref{lem:2.6first}, we now turn to the information encoded by the MC elements of the $L_\infty[1]$-algebra $(\Omega^\bullet(A)[2],\{\mu^B_k\})$.
		Using the identification \eqref{eq:VB-iso_A+A^ast}, the relation
		\begin{equation*}
			L=\gr(\eta)=\{\xi+\iota_\xi\eta\mid\xi\in A\}\subset A\oplus A^\ast\simeq A\oplus B=E.
		\end{equation*}
		yields a correspondence between $2$-forms $\eta\in\Omega^2(A)$ and almost Dirac structures $L\subset E$ close to $A$, in the sense that they are still transverse to $B$.
		As proven in~\cite[Thm.~6.1]{liu1995manin} (when the complement $B$ is Dirac) and~\cite[Lemma 2.6]{FZgeo} (in the general case), the MC elements of $(\Omega^\bullet(A)[2],\{\mu_k^B\})$ are exactly those $2$-forms $\eta\in\Omega^2(A)$ such that the corresponding almost Dirac structure $L$  is involutive, i.e.~Dirac. 
		\begin{lemma}[{\cite[Thm.~6.1]{liu1995manin} and \cite[Lemma 2.6]{FZgeo}}]
			\label{lem:2.6second}
			Let $E$ be a Courant algebroid and $A\subset E$ a Dirac structure. Fix an almost Dirac structure $B\subset E$ complementary to $A$.
			Then the relation
			\[
			L=\gr(\eta)
			\]
			establishes a canonical one-to-one correspondence between
			\begin{enumerate}
				\item MC elements $\eta$ of  $(\Omega^\bullet(A)[2],\{\mu^B_k\})$,
				\item Dirac structures $L\subset E$ transverse to $B$.
			\end{enumerate}
		\end{lemma}

		\bibliographystyle{plain}

	\end{document}